\documentclass[10pt,a4paper]{article}
\usepackage[latin1]{inputenc}
\usepackage[T1]{fontenc}
\usepackage{amssymb,amsthm,amsmath,amsfonts,mathrsfs}
\usepackage{hyperref}
\usepackage{makeidx}
\usepackage{graphicx}
\usepackage[english]{babel}
\usepackage{color}
\usepackage[width=16.4cm, height=24.5cm]{geometry}

\newtheorem{prop}{Proposition}[section]
\newtheorem{lem}{Lemma}[section]
\newtheorem{coro}{Corollary}[section]
\newtheorem{thm}{Theorem}[section]
\newtheorem{remark}{Remark}[section]
\newtheorem{example}{Example}[section]
\numberwithin{equation}{section}

\newcommand{\thesectionwords}{\ifcase \thechapter \fi}

\usepackage{fancyhdr}

\newcommand{\be}{\begin{equation} \label}
\newcommand{\ee}{\end{equation}}
\newcommand{\R}{\mathbb{R}}
\newcommand{\N}{\mathbb{N}}
\newcommand{\eps}{\varepsilon}

\begin{document}
	\begin{center}
	
		\section*{Refined behavior and structural universality \\ 
		of the blow-up profile for the semilinear heat equation \\ 
		with non scale invariant nonlinearity}
	
		Loth Damagui CHABI and Philippe SOUPLET
		\footnote{Universit\'e Sorbonne Paris Nord \& CNRS UMR 7539, Laboratoire Analyse G\'eom\'etrie et Applications, 93430 Villetaneuse,  France.
	Emails: chabi@math.univ-paris13.fr; souplet@math.univ-paris13.fr}
\end{center}

\begin{abstract}
	We consider the semilinear heat equation 
$$
	u_t-\Delta u=f(u) 
$$
	for a large class of non scale invariant nonlinearities of the form 
	$f(u)=u^pL(u)$, where $p>1$ is Sobolev subcritical and $L$ is a slowly varying function
	(which includes for instance logarithms and their powers and iterates, as well as some strongly oscillating functions).
	For any positive radial decreasing blow-up solution,
	we obtain the sharp, global blow-up profile in the scale of the original variables $(x, t)$, which takes the form:
 $$
	u(x,t)=(1+o(1))\,G^{-1}\bigg(T-t+\frac{p-1}{8p}\frac{|x|^2}{|\log |x||}\bigg),
		\ \hbox{as $(x,t)\to (0,T)$, \quad where } G(X)=\int_{X}^{\infty}\frac{ ds}{f(s)}.
$$
This estimate in particular provides the sharp final space profile and the refined space-time profile.
 As a remarkable fact and completely new observation, 
our results reveal a {\it structural universality} of the global blow-up profile, 
being given by the ``resolvent'' $G^{-1}$ of the ODE, composed with a universal,
time-space building block, which is the same as in the pure power case.

\vskip 6pt

\noindent{\bf Key words:} Semilinear heat equation, blow-up profile, refined space-time behavior, slow variation.
\vskip 3pt

\noindent{\bf AMS subject classifications:} 35K58, 35B44, 35B40

\end{abstract}

\section{Introduction}
We consider the semilinear heat equation with general nonlinearity $f$:
	\begin{equation}
	\begin{cases}
	u_t-\Delta u=f(u),&x\in \Omega,\ t>0,\\
	u=0,&x\in \partial\Omega,\ t>0,\\
	u(x,0)=u_0(x),&x \in \Omega,
	\end{cases},\label{eqE1}
	\end{equation}
	where $\Omega \subset \mathbb{R}^n$ $(n\ge 1)$ is smooth and $f\in C^1([0,\infty))$ 
	satisfies $f(0)\ge 0$ and $f(s)>0$ for $s$ large.
	It is well known that, for any $u_0\in L^\infty(\Omega)$, $u_0\ge 0$, problem \eqref{eqE1} admits a unique, nonnegative maximal classical solution.
	Throughout this paper, we shall denote this solution by $u$ 
and its maximal existence time by $T:=T_{\max} (u_0)\in (0,\infty]$.
If $f$ has superlinear growth in the sense that $1/f$ is integrable at infinity then
(see,~e.g.,~\cite[Section~17]{quittner2019superlinear}),
under suitable largeness condition on the initial data, 
$u$ blows up in finite time, i.e.~$T < \infty$ and
	\begin{equation*}
	\lim_{t\to T} \|u(t)\|_\infty=\infty.
	\end{equation*}

In the classical case of the pure power nonlinearity $f(s)=s^p$ ($p>1$), 
very precise information on the asymptotic
behavior of solutions near blow-up points is available.
 In particular the sharp final blow-up profiles 
and the corresponding refined space-time behaviors have been completely classified
 in the Sobolev subcritical range $1<p<p_S$, where 
$$p_S=
\begin{cases}
\frac{n+2}{n-2},&\hbox{ if $n\ge 3$} \\
\noalign{\vskip 1mm}
\infty,&\hbox{ if $n\le 2$}
\end{cases}
$$
(cf.~\cite{filippas1992refined, herrero1992blow, HV93, Vel92, Vel93b, BK94, 
merle1998optimal, merle1998refined, souplet2019simplified} 
and see Remark~\ref{rem1b} 
 below for more details).
However, the analysis in these works heavily depends on the scale invariance
properties of the equation for this specific nonlinearity.
 As already noted in \cite[Section~5.3]{BB}, 
the precise asymptotic blow-up behavior for general nonlinearities 
is still a widely open problem
(see Remark~\ref{remRelated} below for related results).
The main goal of this paper is to partially fill this gap by providing a very precise description
of the refined blow-up behavior, in the case of radially symmetric decreasing solutions, for a large class of non scale invariant nonlinearities.

\section{Main results}

\subsection{Statements of main results}

For $p\in\R$, recall that a function $f$, positive near $\infty$, is said to have regular variation of index $p$ at $\infty$ if $L(s):=s^{-p}f(s)$ satisfies
	\begin{equation}\label{prop}
	\lim_{\lambda\to\infty}\frac{L(\lambda s)}{L(\lambda)}=1\quad \text{for each }\ s>0.
	\end{equation}
The corresponding function $L$ is said to have slow variation at $\infty$.
Functions of regular and slow variation, introduced in \cite{karamata1930mode}, appear in various fields 
(see \cite{seneta} for a general reference and, e.g,~\cite{bingham1989regular} and 
\cite{cirstea, souplet2022universal} 
for applications in probability theory and PDE, respectively).
 When $L$ is $C^1$ near infinity, a well-known sufficient condition for \eqref{prop} is $\lim_{s\to\infty} s\frac{L'(s)}{L(s)}=0$.
 
 Assume that
  \begin{equation}\label{sem0}
 f\in C^1([0,\infty)),\quad f(0)\ge 0, \quad\hbox{ $f$ is positive and $C^2$ for large $s$.} 
  \end{equation}
  For our most precise results, we shall consider the following subclass of 
 regularly varying functions with index $p>1$:
 \begin{equation}\label{sem}
f(s)=s^pL(s), \ \hbox{ where } 
\frac{sL'}{L}=o\Bigl(\frac{1}{\log^\alpha s}\Bigr)
\, \hbox{ and }\, 
\Bigl(\frac{sL'}{L}\Bigr)'=o\Bigl(\frac{1}{s\log s}\Bigr)
\hbox{ as $s\to\infty$, \, for some $\alpha>\frac{1}{2}$.}
 \end{equation}
Such $L$ include logarithms and their powers and iterates, but also functions with intermediate 
growth between logarithms and power functions, as well as some strongly oscillating functions; 
 see Example~\ref{ex1} below for details.

In view of our results we also introduce the function
\begin{equation}\label{so}
G(X)=\int_X^\infty \frac{ ds}{f(s)}.
\end{equation} 
It is easily checked that, under assumptions \eqref{sem0}-\eqref{sem} with $p>1$, $G$  is well defined and is a decreasing bijection from $[a,\infty)$ to $(0,G(a)]$ for some $a>0$.
 Note that the function 
 $$\psi(t):=G^{-1}(T-t),$$ 
 defined for $t$ close to $T^-$, is 
  the unique positive solution of the corresponding ODE 
$\psi'=f(\psi)$ 
which blows up at $t=T$.
  
  Our first main result is the following sharp, global blow-up profile,
  valid in the scale of the original variables $(x, t)$.

 	\smallskip
	
\begin{thm}\label{theo}
	Let $1<p<p_S$, $\Omega=\mathbb{R}^n$ or $\Omega=B_R$, and assume \eqref{sem0}-\eqref{sem}.
	Let $u_0\ge 0$ be radially symmetric, nonincreasing in $|x|$, and assume $T<\infty$.
		If $\Omega=\mathbb{R}^n$ assume also that $u_0$ is nonconstant.
	Then we have
\begin{equation}\label{infer1}
u(x,t)=(1+o(1)) G^{-1}\bigg(T-t+\frac{p-1}{8p}\frac{|x|^2}{|\log|x||}\bigg)\quad \text{as}\ (x,t)\to (0,T).
\end{equation} 
\end{thm}

 	\smallskip
	
 By combining \eqref{infer1} with the available asymptotic formula for $G^{-1}$,
one obtains the explicit blow-up asymptotics of $u$, which shows the 
precise influence of the multiplicative factor $L$ on the profile; see after Example~\ref{ex1}.

On the other hand, under the assumption of Theorem~\ref{theo}, $u\geq0$ is radially symmetric decreasing in $|x|$ for all $t\in (0,T)$ and 
\eqref{infer1} in particular implies that $0$ is the only blow-up point of $u$.
By standard parabolic estimates, it follows that $u$ possesses a final profile, i.e.~the limit $u(x,T):=\lim_{t\to T}u(x,t)$ exists for all $x\ne 0$.
As direct consequence of Theorem~\ref{theo}, 
we obtain a sharp description of the final profile and the refined space-time profile:

\begin{coro}\label{Debo}
	Under the assumptions of Theorem~\ref{theo}, we have 
	
		\begin{itemize}
			\item [(i)]\textbf{(final profile)} \begin{equation}\label{fp}
			u(x,T)=(1+o(1)) G^{-1}\bigg(\frac{(p-1)|x|^2}{8p|\log|x||}\bigg)   \quad\text{as }\ x\to 0;
			\end{equation}
			\item [(ii)] \textbf{(refined space-time profile)}
			\begin{equation}\label{lou}
			u\bigl(\xi\sqrt{(T-t)|\log(T-t)|},t\bigr)= (1+o(1))G^{-1}\bigg[(T-t)\bigg(1+\frac{(p-1)|\xi|^2}{4p}\bigg)\bigg] 
			\quad\text{as }\ t\to T,
			\end{equation}
			uniformly for $\xi$ bounded.
		\end{itemize}
\end{coro}

 When rescricted to backward space-time parabolas, \eqref{lou} in particular contains the fact that 
			\begin{equation}\label{ODEbu}
			\lim_{t\to T} \frac{u(y\sqrt{T-t},t)}{G^{-1}(T-t)}=1, \hbox{ uniformly for $y$ bounded.}
			\end{equation}
This property, which is the extension of the classical 
Giga-Kohn result \cite{giga1985asymptotically,giga1989nondegeneracy}, was recently obtained by the first author in \cite{chabi1}
for general blow-up solutions (without radial symmetry) and a larger class of regularly varying nonlinearities
 (see Theorem~\ref{RDT} in appendix below).
Here, as a second main result in radially symmetric framework, 
we obtain the following theorem, which provides the precise rate of convergence to ODE solution 
in backward space-time parabolas and does not directly follow from Theorem~\ref{theo}.

\begin{thm}\label{Thm1}
	Under the assumptions of Theorem~\ref{theo}, we have 	
	\begin{equation}\label{A1}
	\frac{u(y\sqrt{T-t},t)}{G^{-1}(T-t)}=1-\frac{|y|^2-2n+o(1)}{4p|\log(T-t)|}\quad \text{as} \ t\to T, 
	\end{equation}
	with convergence in $H^1_\rho$ and uniformly for $y$ bounded,
	 where $\rho(y)=e^{-|y|^2/4}$, $\|v\|^2_{H^1_\rho}=\int_{\R^n}(|\nabla v|^2+v^2)\rho dy$
	and $u$ is extended to be $0$ outside $\Omega$ in the case $\Omega=B_R$.
	\end{thm}

\subsection{Discussion and examples}

\begin{remark} \label{rem1}
 {\bf Structural universality of the profile.}
 As a remarkable fact and completely new observation, 
we see from Theorem~\ref{theo} that the global blow-up profile \eqref{infer1} 
 possesses a universal structure, which is to a large extent independent of the nonlinearity.
Namely it is given by the ``resolvent'' $G^{-1}$ of the ODE, composed with a universal,
time-space building block 
	\begin{equation}\label{block}
	T-t+\frac{p-1}{8p}\frac{|x|^2}{|\log|x||}
	\end{equation}
which is the same as in the pure power case.
In particular, the building block of the profile is not affected by the slowly varying factor $L$ of the nonlinearity
and, at principal order, the effect of the latter on the profile is reflected only in the application of the resolvent $G^{-1}$ to this building block.
A similar observation applies for the final and refined space-time profiles and
for the sharp behavior in backward space-time parabolas
(cf.~\eqref{fp}-\eqref{A1}). 
\end{remark}

\begin{example} \label{ex1}
Typical examples of nonlinearities fulfilling assumptions \eqref{sem0}-\eqref{sem},
so that Theorems~\ref{theo}-\ref{Thm1} and Corollary~\ref{Debo} apply,
 are $f(s)=s^pL(s)$ with $L$ given by~:
\be{examplesSlowVar}
\left\{\begin{aligned}
&\hbox{\ $\bullet$\  $\log^a (K+s)$ for $K>1$ and $a\in\R$,} \\
\noalign{\vskip 1mm}
&\hbox{\ $\bullet$\  the iterated logarithms {$\log_m(K+s)$,\footnotemark}} \\
 \noalign{\vskip 1mm}
 &\hbox{\ $\bullet$\  $\exp(|\log s|^\nu)$ with $\nu\in (0,1/2)$,} \\
 \noalign{\vskip 1mm}
 &\hbox{\ $\bullet$\  the strongly oscillating functions} \\
 &\hbox{\qquad $\bigl[\log(3+s)\bigr]^{\sin[\log\log(3+s)]}
\quad\hbox{ and }\quad
\exp\bigl[|\log s|^\nu\cos(|\log s|^\gamma)\bigr],\ \ \nu, \gamma>0,\ \nu+\gamma<1/2$,} \\
\noalign{\vskip 1mm}
 &\hbox{\ $\bullet$\  $1 +a \sin\bigl(\log^\nu(2+s)\bigr)$ with $ \nu\in (0,1/2)$ and $|a|<1$.} \\
\end{aligned}
\right.
\ee
Our results thus cover a large class of non scale invariant nonlinearities.
However, it so far remains an open problem what is the largest possible class of $f$
for which the conclusions of Theorems~\ref{theo}-\ref{Thm1} hold.
\end{example}

\footnotetext{where $\log_m=\log\circ\dots\circ\log$ ($m$ times), $m\in\N^*$
and $K>[\exp\circ\dots\circ\exp](0)$.}

 Under the assumptions on $f$ in Theorem~\ref{theo}, 
we have the following asymptotic formula for $G^{-1}$ (see Lemma~\ref{ell12a}(ii)):
\be{asymptGinv}
G^{-1}(s)=(1+o(1))\kappa s^{-\beta} L^{-\beta}\bigl(s^{-\beta}\bigr)\quad\hbox{as } s\to 0^+,
\quad\hbox{ where } \beta=\frac{1}{p-1},\ \ \kappa=\beta^\beta.
\ee
By combining with \eqref{infer1}, one obtains the explicit formula for the blow-up asymptotics of $u$, 
which can be applied for instance in all the examples above, and displays the 
precise effect of the multiplicative factor $L$ on the profile:
\be{asymptGinv2}
u(x,t)=(1+o(1))\kappa \bigg[T-t+\frac{p-1}{8p}\frac{|x|^2}{|\log|x||}\bigg]^{-\frac{1}{p-1}}
L^{-\frac{1}{p-1}}\bigg(\bigg[T-t+\frac{p-1}{8p}\frac{|x|^2}{|\log|x||}\bigg]^{-\frac{1}{p-1}}\biggr),
\ee
as $(x,t)\to (0,T)$.

\goodbreak

\begin{remark} \label{rem1b}
 {\bf Case of power nonlinearity.}
 In the classical, pure power case $f(s)=s^p$, 
 it was shown in \cite{herrero1992blow, Vel92} that the final blow-up profile~\eqref{fp}
with $G^{-1}(s)=((p-1)s)^{-1/(p-1)}$
and the corresponding refined space-time behavior \eqref{lou} occur whenever the blow-up solution is radially symmetric decreasing.
The corresponding global profile \eqref{infer1} 
was obtained by the second author
in \cite{souplet2019simplified} 
\footnote{In fact, even for $f(s)=s^p$, Theorem~\ref{theo} slightly improves on \cite{souplet2019simplified}
by providing a simpler, equivalent expression of the global profile.}.
 As for the precise rate of convergence \eqref{A1} in backward space-time parabolas
it was obtained in \cite{filippas1992refined, HV93, Vel93b} (see also \cite{FL93, bebernes1992final, Liu93}).

Actually, as a major achievement in the theory of nonlinear parabolic equations, the complete classification of
final blow-up profiles and refined space-time behaviors for $f(s)=s^p$ with $p<p_S$ (without symmetry assumptions)
was obtained in \cite{herrero1992blow, HV93, Vel92, Vel93b, merle1998refined},
as a result of very long, delicate and technical proofs, further developing the pioneering ideas from \cite{filippas1992refined, FL93}.
Among all the possible profiles, \eqref{fp} is the less singular one and it is moreover stable with respect to initial data in a suitable sense 
(see \cite{HV92c, MZ97, FMZ00}).
The results in \cite{herrero1992blow, Vel92} on the radially symmetric decreasing case, mentioned in the previous paragraph, were a by-product of the complete classification
and necessitated its full force.
On the contrary, the proof of the result in \cite{souplet2019simplified} was direct and relied on a simplified approach.

\end{remark}

\begin{remark} \label{remRelated}
(i) {\bf (Related results on blow-up profile)} 
As far as we know, the only previous study of refined blow-up profiles for problem \eqref{eqE1}
with a genuinely non-scale invariant nonlinearity\footnote{Nonlinearities with asymptotic scale invariance as $s\to\infty$ 
can be treated by similar methods as for the case $f(s)=s^p$, see~e.g.~\cite{BK94}.} was carried out  in \cite{duong2018construction},
where the special case of the logarithmic nonlinearity 
\be{logNL}
f(u)=u^p\log^q(2+u^2),\quad p>1, \ q\in\R
\ee
  was considered. 
  For \eqref{eqE1} with $f$ given by \eqref{logNL}, the authors of \cite{duong2018construction} construct
  a special, single-point blow-up solution which follows the asymptotic dynamics~\eqref{lou} and the final profile~\eqref{fp}
  (where the corresponding $G^{-1}$ behaves according to \eqref{asymptGinv} with $L(s)=2^q\log^q s$).
The stability of this blow-up behavior with respect to $L^\infty$ perturbations of initial data is also established.
The results in \cite{duong2018construction} are rather different from ours in nature: 
some special stable solutions, 
which are not necessarily symmetric, are constructed for a special nonlinearity;
on the contrary our results do not assert stability and are restricted to radial decreasing solutions,
 but they apply to any such solution and to a large class of nonlinearities.
The construction in \cite{duong2018construction} relies on rigorous matched asymptotics, reduction of the problem to a finite dimensional one and a topological argument. 
The proofs are thus very different from ours (see the next subsection for an outline of the main ideas of proof of our results).

\smallskip
(ii) {\bf (Related results on blow-up rate)} 
For any regularly varying nonlinearity $f$ of index $p\in(1,p_S)$,
it was recently proved by the second author \cite{souplet2022universal} that 
any positive classical  solution of \eqref{eqE1}
(without symmetry restriction) satisfies the universal, initial and final blow-up estimate:
$$\frac{f(u)}{u}\le C(\Omega,f)\bigl(t^{-1}+(T-t)^{-1}\bigr),\quad x\in \Omega,\ 0<t<T.$$
As a consequence  (cf.~Theorem~\ref{RDT} below and see the proof of~\cite[Theorem~5.1]{chabi1} for details), 
the blow-up rate for problem \eqref{eqE1} with such $f$ is always type I, namely:
\be{typeI}
\limsup_{t\to T}\frac{\|u(t)\|_\infty}{\psi(t)}<\infty.
\ee 
In the special case of the logarithmic nonlinearity \eqref{logNL}, the type I property \eqref{typeI}
was obtained before in \cite{HamzaZaag}.
The methods in \cite{souplet2022universal} and \cite{HamzaZaag} are completely different 
(parabolic Liouville theorems and energy arguments, respectively) and
as a further difference, the result in \cite{HamzaZaag} also applies to sign-changing solutions
 but involves constants depending of the solution instead of universal constants in \cite{souplet2022universal}.
Further parabolic Liouville theorems for non scale invariant nonlinearities and their applications 
to a priori estimates are studied in \cite{QSLiouv}.

For the semilinear heat equation in a bounded domain with rather general nonlinearities
(not necessarily of regular variation),
 but for the strongly restricted class of time increasing solutions,
 type I blow-up was proved in the classical paper \cite[Section~4]{friedman1985blow}.

\smallskip

(iii) {\bf (Other properties)} We refer to, e.g., \cite{friedman1985blow,LRS13,Fuj,FI18,LS20,FI22,FHIL} for results on other aspects of problem \eqref{eqE1}
with non scale invariant nonlinearities, such as local solvability for rough initial data, global existence/non-existence, blow-up set or small diffusion limit.
\end{remark}

\subsection{Ideas of proofs}

We end this section by explaining the main ideas and novelties of our proofs.

\smallskip

Assuming Theorem~\ref{Thm1}, the idea of proof of the lower estimate of Theorem~\ref{theo} is, 
for $t_0$ suitably close to $T$, to use the lower estimate on $u(t_0)$ provided by Theorem~\ref{Thm1}
in the region $|x|=O(\sqrt{T-t_0})$ and propagate it forward in time via a delicate comparison argument.
Although this follows the strategy from \cite{HV93} (see also \cite{souplet2019simplified}),
we need to introduce a number of 
nontrivial new ideas, in order to handle our non-scale invariant nonlinearities under assumptions only at infinity.
In fact, the comparison argument in the proofs in the above works used a further rescaling of the solution,
which is admissible only in case of power-like nonlinearities
(and a similar restriction applies in the alternative approach from \cite{merle1998refined}
without maximum principle, based on Liouville type theorems).
Among other things, we observe that such a rescaling can be avoided thanks to a
suitable modification of the comparison function.
\smallskip

The proof of the upper estimate in Theorem~\ref{theo} (and of Theorem~\ref{theoL} below for general nonlinearities)
is based on a suitable modification of the method in \cite{souplet2019simplified, souplet2022refined} for the cases $f(u)=u^p$ or $e^u$ respectively,
applying the maximum principle to a carefully chosen auxiliary functional and
integrating the resulting differential inequality. Namely, we shall consider
$J:=u_r+\frac{rf(u)}{2(A+\log f(u))}$
where $A > 0$ is a sufficiently large constant. 
Nontrivial difficulties arise in order to treat the case of general nonlinearities $f$, under assumptions at infinity only, 
including the necessity to work on a suitable parabolic subdomain where the solution is large.

\smallskip

Finally, the proof of Theorem~\ref{Thm1} involves two main arguments.
The first one is to rescale the solution by similarity variables and ODE normalization.
Such a rescaling, which extends the seminal idea from \cite{giga1985asymptotically} in the case $f(u)=u^p$,
was applied in \cite{duong2018construction} in the case of a logarithmic nonlinearity
to construct special blow-up solutions with prescribed profile,
and then recently by the first author in \cite{chabi1} for the general case to study 
the local asymptotics of general blow-up solutions.
We here combine it with a key technique of linearization and center manifold type analysis
introduced in \cite{filippas1992refined, HV93, Vel93b} for the case $f(u)=u^p$,
and simplified in \cite{souplet2019simplified} in the radial decreasing framework.
Here again significant difficulties arise in the case of non scale invariant nonlinearities
and the proof is rather long and technical.
 In particular, in the rescaled equation, a number of new terms arise from the slow variation factors.
In order to control them in the linearization process we need, among other things, to derive careful asymptotics of the ODE resolvent.

\smallskip

The outline of the rest of the paper is as follows.
In Section~3, we state results on upper blow-up profile estimates for more general nonlinearities.
Theorem~\ref{theoL}, Proposition~\ref{propp} and the upper part of Theorem~\ref{theo}
are proved in Section~4.
In Section~5 we prove the lower part of Theorem~\ref{theo} assuming Theorem~\ref{Thm1}.
The latter is then proved in Section~6.
Finally  the known facts from \cite{souplet2022universal, chabi1} that we use, as well as some auxiliary results and their proofs, 
are gathered in  Appendix.

\section{Upper estimates for more general nonlinearities}

For more general $f$ which do not satisfy the regular variation assumptions,
although the sharp results in Theorems~\ref{theo}-\ref{Thm1} are no longer available, we are still able to prove a
precise upper global space-time estimate.
In turn, this estimate will provide the upper part of the sharp estimate in Theorem~\ref{theo}.
Namely, we make the following assumptions:
	 \begin{equation}\label{sem0b}
 f\in C^1([0,\infty)),\quad f(0)\ge 0, \quad \lim_{s\to \infty}f(s)=\infty,
 \end{equation}
 \begin{equation}\label{assymption}
	\hbox{$f$ is $C^2$ and } \frac{f}{\log f} \hbox{ is convex at $\infty$,} \quad \int^\infty \frac{\log f(s)}{f(s)}ds <\infty.
	\end{equation}

\begin{thm}\label{theoL}
	Let $\Omega=\mathbb{R}^n \text{ or }\Omega=B_R$, and assume \eqref{sem0b}-\eqref{assymption}.
	Let $u_0\ge 0$ be radially symmetric, nonincreasing in $|x|$ and assume $T<\infty$.
	If $\Omega=\mathbb{R}^n$ assume also that $u_0$ is nonconstant. 
	Then there exist $A,\rho>0$ such that
	\begin{equation}\label{upperTheoL}
	u(x,t)\le H^{-1}\bigg(H(u(0,t))+\frac{|x|^2}{4}\bigg)\quad \text{in}\ B_\rho\times(T-\rho,T),
	\end{equation}
	where
	  \begin{equation}\label{choi}
	H(X)=\int_X^\infty \frac{A+\log f(s)}{f(s)}ds.
	\end{equation}
\end{thm}

Note that, under our assumptions, $H(X)$ is well defined and decreasing for $X$ large and,
since $\lim_{t\to T}$ $u(0,t)=\infty$,
the RHS of \eqref{upperTheoL} is well defined 
for $\rho>0$ small. Estimate \eqref{upperTheoL} implies that
 $0$ is the only blow-up point of $u$, so that $u$ possesses a final blow-up profile (cf.~after Theorem~\ref{theo}).
As consequence of Theorem~\ref{theoL}, we have the following estimate of the final profile.

\begin{coro}[final profile estimate]\label{coro}
	Under the assumptions of Theorem~\ref{theoL} we have 
$$
	u(x,T)\le H^{-1}\bigg(\frac{|x|^2}{4}\bigg)\quad \text{in}\ B_\rho.
$$
\end{coro}

\begin{remark} \label{rem3} 
The integrability condition in \eqref{assymption} is essentially optimal. Indeed global (resp.,~regional) blow-up occurs for $f(s)=s\log^k(2+s)$ with $k<2$ (resp.,~$=2$); cf.~\cite{lacey,GV96}.
\end{remark}

The upper estimate in Theorem~\ref{theo} will be derived from the following 
consequence of Theorem~\ref{theoL} where, for regularly varying nonlinearities, 
the apparently different form of the RHS
(involving $H^{-1}$) can in fact be recast in terms of $G^{-1}$.

\begin{prop}\label{propp}
 Let $\Omega$, $u_0$ be as in Theorem~\ref{theoL},
let $p>1$ and let $f$ satisfy \eqref{sem0b}.
Assume that $f$ is $C^2$ for large $s$ and that $L(s):=s^{-p}f(s)$ satisfies 
\be{hyplm1a}
\lim_{s\to\infty}\frac{sL'(s)}{L(s)}=\lim_{s\to\infty}\frac{s^2L''(s)}{L(s)}=0.
\ee
Then
	\be{conclprop31}
	u(x,t)\le (1+o(1))G^{-1}\bigg(G(u(0,t))+\frac{p-1}{4p}\frac{|x|^2}{\bigl|\log|x|^2\bigr|}\bigg)
	\quad \hbox{as $(x,t)\to (0,T)$,}	
	\ee
 where $G$ is given by \eqref{so}.
\end{prop}

We finally comment on the extension of Theorems~\ref{theo}-\ref{Thm1} 
to the case of critical and supercritical powers.

\begin{remark} \label{rem2} 
Let $f$, $u_0$ be as in Theorem~\ref{theo} with $n\ge 3$. As a consequence of our proofs, 
if $p=p_S$ and if blow-up is type I (cf.~\eqref{typeI}), or if $p>p_S$ and $u$ satisfies \eqref{ODEbu}, then the conclusions of Theorems~\ref{theo}-\ref{Thm1} and Corollary~\ref{Debo} remain valid.
Note that, unlike in the range $p<p_S$ (cf.~Remark~\ref{remRelated}), blow-up need not be type I for $p\ge p_S$ in general 
and that type I blow-up solutions need not satisfy \eqref{ODEbu} for $p>p_S$ (see, e.g., \cite[Section~25]{quittner2019superlinear}).
\end{remark}

\goodbreak 

{
\section{Proof of Theorem~\ref{theoL},  Proposition~\ref{propp}
 and  upper  part of  Theorem~\ref{theo}}
 \label{upperrrrrr}

\subsection{Preliminaries}

The proof of Theorem~\ref{theoL} is based on a suitable modification of the method in \cite{souplet2019simplified, souplet2022refined},
applying the maximum principle to a  carefully chosen auxiliary functional $J$ and 
integrating the resulting differential inequality. Namely, we shall consider
\be{defJlog}
J:=u_r+\frac{rf(u)}{2(A+\log f(u))},
\ee
where $A > 0$ is a sufficiently large constant, defined on a suitable parabolic domain.

\begin{remark} \label{remref}
{\bf (on previous related proofs)} The idea of using functionals of the form $J:=u_r+c(r)F(u)$ to derive upper final blow-up profile estimates
for problem \eqref{eqE1} was introduced in the seminal work \cite{friedman1985blow}.
Nonsharp upper estimates of the final profile in the cases
$f(u)=u^p$ and $f(u)=e^u$ were obtained there, based on the (nonoptimal) choices $c=\eps r$, with $\eps>0$ small, and
$F(u)=u^q$ with $q<p$, resp., $F(u)=e^{ku}$ with $k<1$.

\vskip 1.5pt

In the case $f(u)=u^p$, the technique of \cite{friedman1985blow}
was improved in \cite{GP86}, where the functional $J:=u_r+\frac{ru^p}{2p \log u}$
was used for particular initial data and suitable positive boundary values, and then combined with intersection-comparison
arguments, leading to the sharp upper part of the final space profile estimate for $p < p_S$ under an additional intersection number assumption on $u_0$
(see \cite[Theorem 7.3]{GP86} and final remark in \cite[p.815]{GP86}).
 In the case $f(u)= e^u$, the functional $J:=u_r+{\eps re^u\over 2+u}$ with $\eps>0$ small was used in
 \cite[Corollary~3.17]{bebernes1989mathematical}), leading to a better, but still nonsharp, upper estimate of the final space profile.

\vskip 1.5pt

In the above works, the possibility to establish sharp space-time estimates such as \eqref{infer1}-\eqref{lou}
by this method was not considered.
This was first done in \cite{souplet2019simplified}
in the case $f(u)=u^p$, and 
then in \cite{souplet2022refined} in the case $f(u)=e^u$, through the optimal choices 
$J:=u_r+\frac{ru^p}{2p(A+\log u)}$
and $J:=u_r+\frac{re^u}{2(A+u -\log(A + u))}$, respectively.

\vskip 1.5pt

It turns out that for very general nonlinearities $f$, under a suitable convexity assumption at infinity 
(cf.~\eqref{assymption}), 
sharp space-time estimates can finally be obtained by the choice \eqref{defJlog}, 
combined with a careful analysis involving the parabolic subdomain where $u(x,t)$ is large (see below).
\end{remark}

 We shall need the following lemma, which
collects some useful elementary properties related with our nonlinearities and auxiliary functions.

\begin{lem}\label{lmfF} Under the assumptions of Theorem~\ref{theoL}, there exist $K,M>0$ such that
$f\in C^2([M,\infty))$,
 \be{ffprimeK}
 f, f'\ge -K\quad\hbox{ on $[0,\infty)$}
 \ee
  \be{ffprimeM}
 f\ge e^4,\ f'\ge 0\quad\hbox{ on $[M,\infty).$}
 \ee
Moreover, for any $A\ge 0$, the functions
  \be{defFphiLem}
 F=f\varphi,\quad \varphi=\frac{1}{A+\log f}
\ee
are defined and $C^2$ on $[M,\infty)$ and satisfy
 \be{F00prime}
0<\varphi\le \frac{1}{4},
 \ee
 \be{F0prime}
F'=f'\varphi(1-\varphi),
 \ee
  \begin{equation}\label{supo}
F''\ge 0,
 \end{equation}
  \be{F0primediverge}
 \lim_\infty F'=\infty,\quad  \lim_\infty f'\varphi=\infty.
 \ee
 \end{lem}
\begin{proof}[Proof]
By assumptions \eqref{sem0b}-\eqref{assymption}, there exists $M>0$ such that 
 $f\in C^2([M,\infty))$  and 
 the first part of \eqref{ffprimeM} and \eqref{F00prime} hold.
 Next, in $[M,\infty)$, since $\varphi'=-(f'/f) \varphi^2$, we have
$F'=f'\varphi+f\varphi'=f'(\varphi-\varphi^2)$, i.e.~\eqref{F0prime}.
Denote $F=F_0$ and $\varphi=\varphi_0$ for $A=0$. By assumptions \eqref{sem0b}-\eqref{assymption}, we have $\lim_\infty F_0=\infty$ and we may assume
 \be{F0primeprime}
 F_0''\ge 0 \quad\hbox{ on $[M,\infty)$.}
 \ee
The integrability condition in \eqref{assymption} then implies in particular that 
 \be{F00primediverge}
 \lim_\infty F_0'=\infty,
 \ee
  hence $\lim_\infty f'=\infty$  owing to \eqref{F0prime}. 
By increasing $M$ if necessary, we thus have the second part of \eqref{ffprimeM},  hence \eqref{ffprimeK} for some $K>0$. 
Next observing that \eqref{F0prime} 
implies
$$f'\varphi\ge F'\ge  \frac12 f'\varphi\ge \frac{1}{2(1+A)} f'\varphi_0\ge \frac{1}{2(1+A)} F'_0\quad\hbox{ on $[M,\infty)$,}$$
property \eqref{F0primediverge} follows from \eqref{F00primediverge}.

It remains to show \eqref{supo}. To this end, using \eqref{F0prime} and $\varphi'=-(f'/f) \varphi^2$, we compute
$$ \begin{aligned}
F''&=\bigl(f'\varphi(1-\varphi)\bigr)'=f''\varphi(1-\varphi)+f'\varphi'\bigl(1-2\varphi\bigr) \\
&=f''\varphi(1-\varphi)-\frac{{f'}^2\varphi^2}{f}\bigl(1-2\varphi\bigr) =\frac{\varphi(1-\varphi){f'}^2}{f}\Bigl(\frac{ff''}{{f'}^2}-\frac{\varphi(1-2\varphi)}{1-\varphi}\Bigr).
\end{aligned}$$
Consequently, \eqref{supo} is equivalent to $\frac{ff''}{{f'}^2}\ge h(\varphi)$, where 
$h(s)=\frac{s(1-2s)}{1-s}$. It is easily checked that $h'\ge 0$ on $[0,1/4]$.
It then follows from \eqref{F0primeprime} and $\log f\ge 4$ in $[M,\infty)$ (cf.~\eqref{ffprimeM}) that
$$\frac{ff''}{{f'}^2} \ge  h\Bigl(\frac{1}{\log f}\Bigr)\ge h\Bigl(\frac{1}{A+\log f}\Bigr)=h(\varphi)\quad  \text{ in } [M,\infty),$$
hence \eqref{supo}.
\end{proof}
}

\subsection{Proof of Theorem~\ref{theoL}}\label{sect}

\begin{proof}[Proof]
 We split the proof into several steps.

\smallskip

 \textit{\textbf{Step 1.} Preparations}. 
 Our assumptions guarantee that $u$ is radially symmetric decreasing in $|x|$ for all $t\in (0,T)$.
Setting $r=|x|$, we will sometimes denote $u(r,t)=u(x,t)$ without risk of confusion.

 Set $R=1$ in case $\Omega=\mathbb{R}^n$. 
In view of $f(0)\ge 0$ and \eqref{ffprimeK}, we have 
  \be{linearheatK}
  u_t-\Delta u\ge -Ku,
  \ee
  hence
$u\ge e^{-Kt}e^{t\Delta}u_0$ by the maximum principle, where $(e^{t\Delta})_{t\ge0}$ 
denotes the Dirichlet heat semigroup on $B_R$.
Consequently, there exists $\eta >0$ such that:
$$
 u(x,t)\ge\eta>0\quad \text{in } D:=\bar{B}_{R/2}\times[T/2,T).
$$
 Setting $\Omega_1:=\Omega\cap\{x:x_1>0\}$, we notice that $v:=u_{x_1}\le 0$ satisfies 
 $v_t-\Delta v=f'(u)v\le -Kv$ in $\Omega_1\times(0,T)$, owing to \eqref{ffprimeK}. 
 Therefore $v(t)\le z(t):= e^{-Kt}e^{(t-t_0)A}v(\cdot,t_0)$ in  $\Omega_1\times(t_0,T)$, where $t_0=T/4$ and $e^{tA}$ denotes the Dirichlet heat semigroup on $\Omega_1$. It follows from the strong maximum principle and the Hopf Lemma, applied to $z$, that $v(x_1,0,\cdots,0,t)\le -kx_1$ for all $(x_1,t)\in [0,R/2]\times[T/2,T)$ and some $k>0$. This yields
 \begin{equation}\label{board}
 u_r\le -kr \quad \text{in } [0,R/2]\times [T/2,T).
 \end{equation}

 \textit{\textbf{Step 2.} Parabolic inequality}. 
 With $M$ given by Lemma~\ref{lmfF}, we denote 
 \be{defQQ0}
 Q_0 := (0,R/2) \times (T/2, T),\quad \Sigma=\bigl\{(r,t)\in (0,R)\times (0,T);\ u(r,t)\ge M\bigr\},
\quad Q=Q_0\cap \Sigma.
\ee
We consider an auxiliary function of the form 
$$J := u_r(r, t) + c(r)F(u),$$
where the functions $c$ and $F$, with $F$ defined on $[M,\infty)$, will be chosen below. 
Similar to previous works, we have, in $Q$:
 \begin{equation}\label{oublier}
 c^{-1}\mathcal{P}J=F'f-Ff'+2c'FF'-c^2F''F^2+\Bigl[\frac{n-1}{r^2}\Bigl(1-\frac{rc'}{c}\Bigr)-\frac{c''}{c}\Bigr]F\quad \hbox{in $Q\cap\{c\ne 0\}$}.
 \end{equation}
We give the proof of \eqref{oublier} for completeness. Namely, we compute:
 \begin{equation*}
 \Big(\frac{\partial}{\partial t}-\frac{\partial^2}{\partial r^2} \Big)(cF(u))=cF'(u)(u_t-u_{rr})-cF''(u)u^2_r-2c'F'(u)u_r-c''F(u)
 \end{equation*}
 and
 \begin{equation*}
 \Big(\frac{\partial}{\partial t}-\frac{\partial^2}{\partial r^2} \Big)u_r=\frac{n-1}{r}u_{rr}-\frac{n-1}{r^2}u_r+f'(u)u_r.
 \end{equation*}
 Omitting the variables $r,t,u$ without risk of confusion, it follows that
 \begin{equation*}
 J_t-J_{rr}=\frac{n-1}{r}u_{rr}-\frac{n-1}{r^2}u_r+f'u_r+cF'\big(\frac{n-1}{r}u_r+f\big)-cF''u^2_r-2c'F'u_r-c''F.
 \end{equation*}
 Substituting $u_r=J-cF$ and $u_{rr}=J_r-c'F-cF'u_r=J_r-cF'J+c^2FF'-c'F$, we obtain
 \begin{align*}
 J_t-J_{rr}=&\frac{n-1}{r}\big(J_r-cF'J+c^2FF'-c'F\big)-\frac{n-1}{r^2}(J-cF)+f'(J-cF)\\
 &+cF'\big(\frac{n-1}{r}(J-cF)+f\big)-cF''(J-cF)^2-2c'F'(J-cF)-c''F.
 \end{align*}
 Setting 
 \begin{equation}\label{b}
 \mathcal{P}J:=J_t-J_{rr}-\frac{n-1}{r}J_r+bJ,\quad\text{with } b:=\frac{n-1}{r^2}-f'+cF''(J-2cF)+2c'F'
 \end{equation}
  it follows that
 \begin{equation*}
 \mathcal{P}J=\frac{n-1}{r}(c^2FF'-c'F)+\frac{n-1}{r^2}cF-cFf'+cF'\big(-\frac{n-1}{r}cF+f\big)-c^3F''F^2+2cc'FF'-c''F,
 \end{equation*}
hence \eqref{oublier}.
 
Now, for the rather general class of nonlinearities under consideration,
 an essentially optimal choice of $F$ will turn out to be
 \be{defFphi}
 F=f\varphi,\quad \varphi=\frac{1}{A+\log f}>0\quad  \text{ in } [M,\infty).
\ee
With this $F$, we have $F'f-Ff'=-ff'\varphi^2$ and $F''\ge 0$ owing to \eqref{F0prime}-\eqref{supo}
and we thus obtain
 \be{cfPJ}
    (cf\varphi)^{-1}\mathcal{P}J\le -f'\varphi\Bigl(1-2c'(1-\varphi)\Bigr)+\Bigl[\frac{n-1}{r^2}\Bigl(1-\frac{rc'}{c}\Bigr)-\frac{c''}{c}\Bigr]\ \hbox{ in $Q\cap\{c\ne 0\}$}.
\ee
Different choices of the constant $A\ge 0$ and of the function $c(r)$ will be made in the subsequent two steps. 
The first choice (cf.~Step~3) will be to guarantee single-point blow-up.
 Once this is proved, we will use (cf.~Step~4) a second, more precise choice, to obtain the desired sharp upper bound.

\medskip

  \textit{\textbf{Step 3.} Single-point blow-up.}
Although a similar property was proved in \cite{friedman1985blow}, 
the result there requires global (instead of asymptotic) hypotheses on $f$ 
and does not apply to $\Omega=\R^n$. 
We thus provide a detailed proof which fits our assumptions and framework.
 
  Assume for contradiction that there exists a blow-up point $r\ne 0$, say $r=3a\in(0,R/2)$. 
We first claim that, under this assumption,
\be{ClaimFMLimproved-Rn}
\min_{|x|\le a}u(x,t)\to \infty,\quad\hbox{as $t\to T$.}
\ee
Since $u$ is radially symmetric decreasing, there exists a sequence $t_j\to T$ such that $u(2a,t_j)\to\infty$.
Denoting by $\phi>0$ 
the first eigenfunction of the negative Dirichlet Laplacian in
$B_{2a}$ normalized by $\|\phi\|_\infty=1$ and $\lambda$ the corresponding eigenvalue,
it follows from \eqref{linearheatK} and the maximum principle that 
$u(x,t_j+t)\ge j e^{-(\lambda+K) t} \phi(x)$ in $B_{2a}$ for all $t\in (0,T-t_j)$.
Claim \eqref{ClaimFMLimproved-Rn} follows.

  We now take $F=F_0$ to be defined by \eqref{defFphi} with $A=0$.
  As for $c$, following \cite{FC88} (see also \cite{CM89}), we take
\be{Choice-c}
c(r)=\epsilon\sin^2(\pi r/a)
\ee
with $\epsilon>0$ to be chosen.
(Note that this choice, different from \cite{friedman1985blow}, has the advantage to cover both cases $\Omega=B_R$ and $\Omega=\R^n$.)
 We have
  $$\zeta(r):=\frac{n-1}{r^2}\Bigl(1-\frac{rc'}{c}\Bigr)-\frac{c''}{c}=
  \frac{n-1}{r^2}\Bigl(1-\frac{2\pi r}{a}\cot\Bigl(\frac{\pi r}{a}\Bigr)\Bigr)+2\Bigl(\frac{\pi}{a}\Bigr)^2\Bigl(1-\cot^2\Bigl(\frac{\pi r}{a}\Bigr)\Bigr),\quad 0<r<a,$$
  hence $\lim_{r\to 0^+} \zeta(r)= -\infty$, so that $K_1:=\sup_{r\in(0,a)} \zeta(r)<\infty$.
Also, by \eqref{ClaimFMLimproved-Rn} there exists $\eta>0$ such that $u\ge M$ in $(0,a)\times (T-\eta,T)$.
Taking $\epsilon$ small enough so that $\|c'\|_\infty\le 2\epsilon\pi a^{-1}\le 1/4$, 
 and using \eqref{ffprimeM}, \eqref{F00prime},  it follows from \eqref{cfPJ} that
  $$    (cf\varphi)^{-1}\mathcal{P}J\le -f'\varphi\bigl(1-2 |c'|\bigr)+K_1
  \le -\frac{1}{2}f'\varphi+K_1\quad\hbox{in $(0,a)\times (T-\eta,T)$.}$$
Since $\lim_\infty f'\varphi=\infty$ due to \eqref{F0primediverge},
by taking $\eta$ smaller if necessary (independent of $\epsilon$), it follows from \eqref{ClaimFMLimproved-Rn} 
that 
$$\mathcal{P}J\le 0\quad\hbox{in $(0,a)\times (T-\eta,T)$.}$$
Moreover, $J(0,t)=0$ owing to $u_r(0,t)=c(0)=0$ and $J(a,t)\le 0$ owing to $u_r\le 0$ and $c(a)=0$.
Taking $\epsilon$ possibly smaller and using \eqref{board}, we 
also have
$$J(r,T-\eta)\le \bigl\{-k +\epsilon\textstyle\frac{\pi^2}{a}\|F(u(\cdot,T-\eta))\|_{L^\infty(0,a)}\bigr\}r\le0\quad\hbox{in $[0,a]$.}$$
  Since the coefficient $b$ in \eqref{b} is bounded from below for $t$ bounded away from $T$, we may apply the maximum principle to deduce that $J \leq  0$ in $(0,a)\times (T-\eta,T)$,~i.e:
$$
 -\frac{u_r \log f(u)}{f(u)}\ge c(r)\quad\hbox{in $(0,a)\times (T-\eta,T)$.}
$$
By integration, recalling \eqref{assymption}, we deduce that 
$$\int_{u(a,t)}^\infty \frac{\log f(s)}{f(s)}ds\ge \int_0^ac(r)dr>0,\quad T-\eta<t<T,$$
hence $\sup_{t\in (T-\eta,T)}u(a,t)<\infty$. But this contradicts~\eqref{ClaimFMLimproved-Rn}.
We have proved that $0$ is the only blow-up point.

  \bigskip
 
 \textit{\textbf{Step 4.} Improved choice of auxiliary functions.}
  Now choose 
  $$c(r)=r/2,$$
  hence $c-rc'=c''=0$,
and let $F$ be given by \eqref{defFphi} with 
 \be{PJcontrad2A}
 A=k^{-1}\max\Bigl\{(\|f(u(\cdot,T/2))\|_\infty, f(M), \sup_{t\in(T/2,T)} f(u(R/2,t))\Bigr\}
 \ee
(note that the supremum in \eqref{PJcontrad2A} is finite owing to the fact that $0$ the only blow-up point).
By \eqref{cfPJ} we have:
 \be{PJneg}
    (cf\varphi)^{-1}\mathcal{P}J\le -f'\varphi\Bigl(1-2c'(1-\varphi)\Bigr)=-f'\varphi^2<0
\quad  \text{ in } Q.
\ee
We claim that
 \be{PJcontrad0}
 J\le 0 \quad  \text{ in } Q=Q_0\cap \Sigma.
 \ee
Since the set $Q$ involves the constraint $u>M$, some care is needed and we detail
the maximum principle argument to make everything safe.

Assume for contradiction that \eqref{PJcontrad0} fails, then there exists $\tau\in(T/2,T)$ such that
  \be{PJcontrad1}
 \max_{\overline Q \cap[T/2,\tau]} J>0.
 \ee
We split $\bar Q \cap[T/2,\tau]=\displaystyle\bigcup_{i=1}^5 D_i$, where
$$
D_1=\bigl(\{0\}\times [T/2,\tau]\bigr)\cap \Sigma,\quad
D_2=\bigl(\{R/2\}\times [T/2,\tau]\bigr)\cap \Sigma, \quad
D_3=\bigl([0,R/2]\times\{T/2\}\bigr)\cap \Sigma,
$$
$$
D_4=\bigl((0,R/2)\times (T/2,\tau]\bigr)\cap \Sigma,\quad
D_5=\bigl([0,R/2]\times [T/2,\tau]\bigr)\cap \partial\Sigma.
$$
We have 
 \be{PJcontrad2}
 J=0 \quad \text{on } D_1,
 \ee
 owing to $u_r(0,t)=c(0)=0$. Using \eqref{board}, \eqref{PJcontrad2A} and the fact that $u=M$ on $\partial\Sigma$, we have
 \be{PJcontrad3}
J\le \Bigl\{-k+\frac{f(u)}{2\big(A+\log f(u)\big)}\Bigr\}r
\le 0 \quad \text{in } D_2\cup D_3\cup D_5.
\ee
Next set
$$ \tilde J:=e^{Lt} J,\quad\hbox{where } L=L_\tau=-1+\inf\bigl\{b(x,t);\ (x,t)\in (0,R/2]\times [T/2,\tau]\bigr\}>-\infty.$$
It follows from \eqref{PJcontrad2}-\eqref{PJcontrad3} that $\tilde J\le 0$ in $D_1\cup D_2\cup D_3\cup D_5$.
By \eqref{PJcontrad1} we deduce that $\tilde J$ attains its positive maximum on $\bar Q$ at some point $X_0=(x_0,t_0)\in D_4$.
Since $\Sigma$ is open, there exists $\delta>0$ such that 
$[x_0-\delta,x_0+\delta]\times[t_0-\delta,t_0] \in \bar Q$
and at the point $X_0$, we thus have ${\tilde J}_r=0$, ${\tilde J}_{rr}\le 0$ and ${\tilde J}_t\ge 0$. 
In view of \eqref{PJneg}, this implies that, at the point $X_0$,
$$0\le {\tilde J}_t-{\tilde J}_{rr}-\frac{n-1}{r}{\tilde J}_r
=e^{Lt}\bigl(J_t-J_{rr}-\frac{n-1}{r}J_r+LJ\bigr)
=e^{Lt}\bigl(\mathcal{P}J+(L-b)J\bigr)<-e^{Lt}J<0.$$
 This is a contradiction and \eqref{PJcontrad0} is proved.

\medskip

 \textit{\textbf{Step 5.} Integration and conclusion}. Since $\lim_{t\to T} u(0,t)=\infty$ and $u$ is radially symmetric decreasing, 
 we may find $\eta\in(0,T/2)$ such that,
 for each $t\in (T-\eta,T)$, there exists $r_0(t)\in (0,R/2]$ 
 such that $u(r,t)\ge M$ on $(0,r(t)]$ and $u(r,t)<M$ on $(r(t),R/2]$ (this last interval possibly empty).
 Fixing $t\in(T-\eta,T)$, \eqref{PJcontrad0} yields
  \be{Lowerur}
  -u_r\ge \frac{rf(u)}{2\big(A+\log f(u)\big)},\quad 0<r<r(t)
  \ee
and, by integration, we obtain, for all $r\in (0,r(t)]$:
$$
 \frac{r^2}{4}\le-\int_{0}^{r}\frac{u_r(s)\big(A+\log f(u(s))\big)}{f(u(s))}ds  
 = \int_{u(r,t)}^{u(0,t)}\frac{A+\log f(y)}{f(y)}dy= H(u(r,t))-H(u(0,t)),
$$
 hence
$$
 H(u(r,t))\ge H(u(0,t))+\frac{r^2}{4},\quad 0<r<r(t).
$$
Since $H$ is a decreasing bijection from $[M,\infty)$ to $(0,H(M)]$ and $H^{-1}$ is also decreasing, we deduce that
 \begin{eqnarray}\label{Lothdam}
 u(r,t)\le H^{-1}\bigg(H(u(0,t))+\frac{r^2}{4}\bigg),\quad 0<r<r(t).
 \end{eqnarray}
 Moreover, for $\rho>0$ small, we have $H(u(0,t))+\frac{r^2}{4}\le H(M)$,
 hence
 $H^{-1}\big(H(u(0,t))+\frac{r^2}{4}\big)\ge M$, in $B_\rho\times(T-\rho,T)$,
 so that \eqref{Lothdam} remains true for $r\in(r(t),\rho]$ (where $u\le M$) and \eqref{upperTheoL}  is proved.
\end{proof}

 \begin{proof}[Proof of Corollary~\ref{coro}]
 Since $u(0,t)\to\infty$ as $t\to T$, we have $H(u(0,T))=0$. Then from \eqref{Lothdam} we 
 deduce
$$  u(r,T)\le H^{-1}\Big(\frac{r^2}{4}\Big) \quad \text{as } r\to0.$$
\end{proof}

 \subsection{Proof of upper part of Theorem~\ref{theo}} 
 
We first prove Proposition~\ref{propp}, from which the upper estimates in Theorem~\ref{theo} will follow easily.
 In order to deduce Proposition~\ref{propp} from Theorem~\ref{theoL},
we shall use the following lemma, which provides useful properties of $G, H$
and especially the relationships between $H$ and $G$
and between $H^{-1}$ and $G^{-1}$
when $f$ is regularly varying.
Here and in the rest of the paper, the asymptotic notation $\sim$
means that the quotient of the two functions converges to $1$.

\begin{lem}\label{lm1}  
 Let $p>1$, let $f$ be $C^1$ and positive for large $s$ and assume that $L(s):=s^{-p}f(s)$ satisfies 
\be{hyplm1}
\lim_{s\to\infty}\frac{sL'(s)}{L(s)}=0.
\ee
(i) As $X\to\infty$, we have
		\begin{align} 
			\label{ch}
		G(X)&\sim \frac{X^{1-p}}{(p-1)L(X)}, \\ 
		\label{ch1}
		H(X)&\sim \frac{pX^{1-p}\log X}{(p-1)L(X)}, \\ 
		\label{ch2}
	H(X)&\sim\frac{ p}{p-1}G(X)|\log(G(X))|. 
	\end{align}
	 (ii) If a function $\eps$ satisfies $\lim_{X\to\infty}\eps(X)=0$
	(resp., $\lim_{Y\to 0^+}\eps(Y)=0$) then
\be{equivGG}
\lim_{X\to\infty}\frac{G\bigl((1+\eps(X))X\bigr)}{G(X)}=1
\quad\Bigr(\hbox{resp.,} \lim_{Y\to 0^+}\frac{G^{-1}\bigl((1+\eps(Y))Y\bigr)}{G^{-1}(Y)}=1\Bigr)
\ee
and the similar properties hold for $H, H^{-1}$.
\smallskip

\noindent(iii) We have
	\begin{equation}\label{dam}
	H^{-1}(Y)\sim G^{-1}\bigg(\frac{(p-1)Y}{p|\log Y|}\bigg) \quad \text{as } Y\to 0^+.
	\end{equation}
\end{lem}

 In order not to interrupt the main line of arguments,
we postpone the proof to  the Appendix.

 \begin{proof}[Proof of Proposition~\ref{propp}]
 Let us check that the hypothesis \eqref{assymption} of Theorem~\ref{theoL} is satisfied.
By assumption \eqref{hyplm1a}, we have $f'(s)\sim ps^{p-1}L(s)$, $f''(s)\sim p(p-1)s^{p-2}L(s)$
and $\frac{{f'}^2(s)}{f(s)}\sim p^2s^{p-2}L(s)$ as $s\to\infty$.
 Let $s_0>0$ be such that $L(s)>0$ and $f(s)\ge 2$ for all $s\ge s_0$.
For $s\ge s_0$, setting $\tilde f(s)=\frac{f(s)}{\log f(s)}$, we compute
$\tilde f'(s)=f'(s)\bigl(\frac{1}{\log f(s)}-\frac{1}{\log^2 f(s)}\bigr)$
and then
$$\tilde f''(s)=f''(s)\Bigl(\frac{1}{\log f(s)}-\frac{1}{\log^2 f(s)}\Bigr)
+\frac{{f'}^2(s)}{f(s)}\Bigl(-\frac{1}{\log^2 f(s)}+\frac{2}{\log^3 f(s)}\Bigr)
\sim p(p-1)s^{p-2}\frac{L(s)}{\log f(s)}$$
as $s\to\infty$. It follows that
$\tilde f(s)$ is convex for large $s$.
  Also, since \eqref{hyplm1a} implies $\log f(s)\sim_\infty p\log s$ and, for any $\eta>0$,
 $(s^{\eta/2} L(s))'>0$ for $s$ large hence $s^\eta L(s)\to \infty$ as $s\to\infty$,
 the integral in \eqref{assymption} is convergent.

\smallskip

We may thus apply Theorem~\ref{theoL} 
and we arrive at 
$$
u(x,t)\le H^{-1}\left(H(u(0,t))+\frac{|x|^2}{4}\right)\quad \text{in } B_{\rho}\times[T-\rho,T).
$$
Since $m(t):=u(0,t)=\|u(t)\|_\infty\to \infty$,  hence $G(m(t))\to 0$, as $t\to T$, 
applying Lemma~\ref{lm1}(i), we get
$$u(x,t)\le H^{-1}\left(H(m)+\frac{|x|^2}{4}\right)
 =H^{-1}\left((1+o(1))\left(\frac{ p}{p-1}G(m)|\log(G(m))|+\frac{|x|^2}{4}\right)\right)$$
as $(x,t)\to (0,T)$.  Next applying Lemma~\ref{lm1}(ii) and (iii), we obtain
\be{nuit}
u(x,t)\le (1+o(1))G^{-1}\left(\frac{G(m)|\log(G(m))|+\frac{p-1}{p}\frac{|x|^2}{4}}
{\bigl|\log\bigl(\frac{p}{p-1}G(m)|\log(G(m))|+\frac{|x|^2}{4}\bigr)\bigr|}\right).
\ee
Also,  for $(x,t)$ close enough to $(0,T)$,
we have $G(m)\le \frac{p}{p-1}G(m)|\log(G(m))|+\frac{|x|^2}{4}<1$,
so that
$$\begin{aligned}
\bigg|\log\bigg(\frac{p}{p-1}G(m)|\log G(m)|+\frac{|x|^2}{4}\bigg)\bigg|
&\le\min\left\{|\log(G(m))|,\bigl|\log\bigl(\textstyle\frac{|x|^2}{4}\bigr)\bigr|\right\}\\
&=(1+o(1))\min\left\{|\log(G(m))|,\bigl|\log|x|^2\bigr|\right\}
\end{aligned}$$
as $(x,t)\to (0,T)$, hence
$$\frac{G(m)|\log(G(m))|+\frac{p-1}{p}\frac{|x|^2}{4}}
{\bigl|\log\bigl(\frac{p}{p-1}G(m)|\log(G(m))|+\frac{|x|^2}{4}\bigr)\bigr|}
\ge (1+o(1)\left(G(m)+\frac{p-1}{4p}\frac{|x|^2}{\bigl|\log|x|^2\bigr|}\right).$$
 Since $G^{-1}$ is nonincreasing, we thus deduce from \eqref{nuit} that
$$
u(x,t)\le(1+o(1)) G^{-1}\left((1+o(1)\left(G(m)+\frac{p-1}{4p}\frac{|x|^2}{\bigl|\log|x|^2\bigr|}\right)\right),
\quad (x,t)\to (0,T).
$$
 Applying Lemma~\ref{lm1}(ii) once again we reach the desired conclusion \eqref{conclprop31}.
\end{proof}

\begin{proof}[Proof of upper estimates in Theorem~\ref{theo}] 
 By \eqref{jojo0} in Theorem~\ref{RDT},  
 we know that
$$\lim_{t\to T}\frac{m(t)}{G^{-1}(T-t)} =1,$$
where $m(t)=u(0,t)$ and, by \eqref{equivGG}, this guarantees that
\be{J1}
G(m(t))=(1+o(1))(T-t),\quad t\to T.
\ee
Combining \eqref{conclprop31} in Proposition~\ref{propp} with \eqref{J1} and applying Lemma~\ref{lm1}(ii), 
we obtain the upper part of~\eqref{infer1}.
\end{proof}

\section{Proof of lower part of Theorem~\ref{theo} assuming Theorem~\ref{Thm1} }

 For $\delta>0$ to be chosen below, we denote by $(S_\delta(t))_{t\ge 0}$ the Dirichlet heat semigroup on $B_\delta\subset\R^n$.
We shall use the following simple lemma.

\begin{lem}\label{lemSrho} 
(i) There exists a constant $\lambda=\lambda(n,\delta)>0$ such that, for any $K,N>0$,
\be{auxilZ1Z2}
\Big[S_\delta(t) \bigl(N-K|x|^2\bigr)\Big]_+\ge e^{-\lambda t}\big[N-K(|x|^2+2nt)\big]_+
\ \hbox{ in $Q:=\overline B_{\delta/2}\times[0,\infty)$.}
\ee

(ii) For all $\phi\in L^\infty(B_\delta)$, we have
\be{auxilHeat}
\bigl|(S_\delta(t)\phi)(x)\bigr|\le C(n)t^{-n/4}\bigg(\int_{B_\delta}|\phi(z)|^2e^{-\frac{|z|^2}{4t}}dz\bigg)^{1/2}e^{\frac{|x|^2}{2t}},\quad x\in B_\delta, \ t>0.
\end{equation}
 \end{lem}
 
\begin{proof}
(i) Set 
$$Z=Z_1Z_2,\ \hbox{ where } Z_1(x,t)=S_\delta(t) \chi_{B_\delta},\quad Z_2(x,t)= N-K(|x|^2+2nt),
\quad \hbox{for all } (x,t)\in \overline B_\delta\times[0,\infty).$$
It is well known that $Z_1(\cdot,t)\ge 0$ is radially symmetric decreasing for each $t>0$
(see, e.g., \cite[Proposition~52.17*(ii)]{quittner2019superlinear}).
Then, using $\partial_t Z_i-\Delta Z_i=0$, we obtain 
$$Z_t-\Delta Z=Z_1\bigl(\partial_t Z_2-\Delta Z_2\bigr)+Z_2\bigl(\partial_t Z_1-\Delta Z_1\bigr)-2\nabla Z_1\cdot\nabla Z_2
=4Kx\cdot\nabla Z_1\le 0$$ 
in $B_\delta\times[0,\infty)$, with $Z(0)=N-K|x|^2$ and $Z_{|\partial B_\delta}=0$. Therefore, by the maximum principle, we get
$S_\delta(t)\big(N-K|x|^2\big)\ge Z(t)$,
hence $\bigl[S_\delta(t)\big(N-K|x|^2\bigr)\bigr]_+\ge Z_+(t)= Z_1(t)[Z_2(t)]_+$.
\smallskip

Let now $\lambda$ be the first Dirichlet eigenvalue of $-\Delta$ in $B_{\delta/2}$, 
and $\phi$ be the corresponding eigenfunction, normalized by $\phi(0)=1$.
Since $\phi$ is radially symmetric decreasing, we have $\max_{\overline B_{\delta/2}}\phi=1$.
For each $x_0\in \overline B_{\delta/2}$, set $\omega_{x_0}=B_{\delta/2}(x_0)\subset B_\delta$ and 
$\underline Z_1:= e^{-\lambda t}\phi(x-x_0)$,
which satisfies $\partial_t \underline Z_1-\Delta \underline Z_1=0$ in $\omega_{x_0}\times (0,\infty)$,
with $\underline Z_1\le Z_1$ on the parabolic boundary.
It follows from the maximum principle
 that $Z_1\ge\underline Z_1$ in $\omega_{x_0}\times (0,\infty)$, hence in particular
$Z_1(t,x_0)\ge e^{-\lambda t}$.
This proves \eqref{auxilZ1Z2}.

\smallskip

(ii) By standard comparison with the Gaussian heat kernel, we have 
$\bigl|(S_\delta(t)\phi)(x)\bigr|\le(4\pi t)^{-n/2}\int_{B_\delta}|\phi(z)|$ $e^{-|x-z|^2/4t}dz$.
On the other hand, using the identity 
$\frac{1}{4}|x-z|^2+\frac{1}{2}|x|^2-\frac{1}{6}|z|^2=\frac{1}{12}|z-3x|^2\ge0$
and the Cauchy-Schwarz inequality, we obtain,
$$
\int_{B_\delta}|\phi(z)|e^{-\frac{|x-z|^2}{4t}}dz\le e^{\frac{|x|^2}{2t}}\int_{B_\delta}|\phi(z)|e^{-\frac{|z|^2}{6t}}dz
\le Ct^{n/4}e^{\frac{|x|^2}{2t}}\bigg(\int_{B_\delta}|\phi(z)|^2e^{-\frac{|z|^2}{4t}}dz\bigg)^{1/2},
$$
hence \eqref{auxilHeat}.
\end{proof}

The idea of proof of the lower estimate of Theorem~\ref{theo} is, 
for $t_0$ suitably close to $T$, to use the lower estimate on $u(t_0)$ provided by Theorem~\ref{Thm1}
in the region $|x|=O(\sqrt{T-t_0})$ and propagate it forward in time via a comparison argument.
Although this follows the strategy from \cite{HV93} and \cite{souplet2019simplified}, 
we need to introduce a number of 
nontrivial new ideas, in order to handle our non-scale invariant nonlinearities under assumptions only at infinity.
In particular, instead of rescaling the solution near the blow-up time like in  \cite{HV93,souplet2019simplified}
(see \cite[Lemma~6.1]{HV93} and \cite[formula (4.1)]{souplet2019simplified}; which can be done only for power-like nonlinearities),
we manage (cf.~Steps~2 and 3) to directly apply a comparison argument on the original equation
by relying on Lemma~\ref{lemSrho}.

\begin{proof}[Proof of lower estimate of Theorem~\ref{theo} assuming Theorem~\ref{Thm1}]
We divide the proof into several steps for clarity.
\smallskip

 {\bf Step 1.} {\it Preparations.} Noting that $\frac{s^2L''}{L} = s \zeta'- \zeta + \zeta^2$ with $\zeta=\frac{sL'}{L}$,
we deduce from assumptions \eqref{sem0}-\eqref{sem} that
 \be{limLprime}
\lim_{s\to\infty}\frac{sL'(s)}{L(s)}=\lim_{s\to\infty}\frac{s^2L''(s)}{L(s)}=0.
\ee
Consequently, $f'(s)\sim ps^{p-1}L(s)$ and $f''(s)\sim p(p-1)s^{p-2}L(s)$ as $s\to\infty$, and there exists $M>0$ such that 
\be{fposM}
f\in C^2([M,\infty))\quad\hbox{and}\quad f,f',f''>0\ \hbox{ on $[M,\infty)$}
\ee 
and $G:[M,\infty)\to(0,G(M)]$
is decreasing and $C^2$.
On the other hand, since $0$ is a blow-up point, 
 by  \eqref{noneedle} in Theorem~\ref{RDT}, 
there exists $\delta\in(0,1)$ with  $\delta<\min\{\sqrt{T},G(M)\}$,
with $\delta<R/4$ if $\Omega=B_R$, such that
\be{uparabbdry}
u(x,t)\ge A:=G^{-1}\big(G(M)-\delta\big)>M \ \hbox{ in $Q:=\overline B_\delta\times[T-\delta^2,T)$.}
\ee

{\bf Step 2.} {\it Comparison argument.}
Fix any $\sigma \in(0,\delta^2)$ (to be chosen later) and set 
\be{deft0}
t_0=T-\sigma.
\ee
Denoting $S=S_\delta$, we introduce the following comparison 
function:\footnote{We note that similar comparison functions have been used for various purposes in 
a number of works on semilinear parabolic equations (see, e.g.,~\cite{KP66, W81, Meier, HV93, Fuj}).}
\be{defUV}
U=U_\sigma(x,t)= G^{-1}\Bigl\{G(V)+t_0-t\Bigr\},
\quad\hbox{ with }V(x,t):=M+S(t-t_0)\bigl(u(t_0)-M\bigr).
\ee
 Since $u(t_0)\ge M$, we have 
$0\le S(t-t_0)\bigl(u(t_0)-M\bigr)\le \|u(t_0)\|_\infty-M$, so that $M\le V\le \|u(t_0)\|_\infty$,
hence $\tau_0:=G(\|u(t_0)\|_\infty)\le G(V)\le G(M)$.
Letting 
\be{deftau1}
E:=\bigl\{\tau>t_0\,;\,G(V)+t_0-t>0 \text{ in } \overline B_\delta\times [t_0,\tau)\bigr\},\quad \tau_1=\sup E,
\quad Q_1:=\overline B_\delta \times [t_0, \tau_1),
\ee
we see that $E$ is nonempty and $\tau_1\ge t_0+\tau_0$, $U$ is well defined and smooth in $Q_1$ and
\be{UVcomp}
U\ge V\ge M \ \hbox{ in $Q_1$.}
\ee

We claim that $\tau_1\ge T$ and that
 \begin{equation}\label{qz}
u\ge U\quad \text{in } B_\delta\times [T-\sigma,T).
\end{equation}
To this end we check that $U$ is a subsolution. We compute, for all $(x,t)\in Q_1$,
 $$U_t=\big(-1+V_t G'(V)\big)(G^{-1})'\Bigl(G(V)+t_0-t\Bigl),
 \quad \nabla U=\big(G'(V)\nabla V \big)(G^{-1})'\Bigl(G(V)+t_0-t\Bigl)$$
  and 
 \begin{align*}
 \Delta U&=(\Delta V) G'(V)(G^{-1})'\Bigl(G(V)+t_0-t\Bigl)\\
 &\quad +|\nabla V|^2\bigg\{G''(V)(G^{-1})'\Bigl(G(V)+t_0-t\Bigl)+(G'(V))^2(G^{-1})''\Bigl(G(V)+t_0-t\Bigl)\bigg\}.
 \end{align*} 
Since $U\ge V\ge M$, we have 
 \begin{equation*}
 G'(V)=-\frac{1}{f(V)},\quad  
  G''(V)=\frac{f'(V)}{f^2(V)},\quad
  (G^{-1})'\Bigl(G(V)+t_0-t\Bigl) =(G^{-1})'\bigl(G(U)\bigl)\ =\frac{1}{G'(U)}=-f(U) 
  \end{equation*}
 and
 $$(G^{-1})''\Bigl(G(V)+t_0-t\Bigl)  = (G^{-1})''\bigl(G(U)\bigl)=-\frac{G''(U)}{{G'}^3(U)}=\frac{f'(U)}{f^2(U)}f^3(U)=f(U)f'(U).$$
 Combining with $V_t-\Delta V=0$,  we obtain
$$
 U_t-\Delta U= f(U)-|\nabla V|^2\bigg(-\frac{f'(V)f(U)}{f^2(V)}+\frac{f'(U)f(U)}{f^2(V)}\bigg)
 =f(U)+\frac{|\nabla V|^2f(U)}{f^2(V)}\big({f'(V)}-f'(U)\big).
$$
 By \eqref{fposM} and \eqref{UVcomp}, we deduce that
$$\partial_tU-\Delta U\le f(U) \quad\hbox{in $Q_1$.}$$
Set $T_1=\min(T,\tau_1)$. On $\partial B_\delta\times[t_0,T_1)$
we have $V=M$, hence $G(V)+t_0-t\ge G(M)-\delta$, so that $U \le G^{-1}\big(G(M)-\delta\big)=A$.
Since also $U(t_0)=u(t_0)$, it follows from the comparison principle that 
$u\ge U$ in $B_\delta\times[t_0,T_1)$.
In particular, we cannot have $T>T_1=\tau_1$ since otherwise, recalling \eqref{deftau1}, we would get
$\inf_{Q_1}(G(V)+t_0-t)=0$, hence $\sup_{Q_1} u\ge \sup_{Q_1}U=\infty$:
a contradiction with $T$ being the blow-up time of $u$. Property \eqref{qz} follows.

\smallskip
{\bf Step 3.} {\it Propagation of lower estimate on $u(t_0)$ in the region $|x|=O(\sqrt{T-t_0})$.}
We use the notation
\begin{equation}\label{defL2H1}
L^2_\rho:=\Bigl\{v\in L^2_{loc}(\mathbb{R}^n);\ \int_{\mathbb{R}^n}v^2(y)\rho(y)dy<\infty\Bigr\},\quad 
H^1_\rho:=\bigl\{f\in L^2_\rho;\ \nabla f\in L^2_\rho\bigr\},\quad \rho(y)=e^{-|y|^2/4},
\end{equation}
and $u$ is defined to be $0$ outside $\Omega$ in case $\Omega=B_R$.
By Theorem~\ref{Thm1}, we have
$$	\frac{u(y\sqrt{\sigma},T-\sigma)}{G^{-1}(\sigma)}=1-\frac{|y|^2-2n}{4p|\log\sigma|}
+\frac{\mathcal{R}(y,|\log \sigma|)}{|\log\sigma|}$$
with
$\lim_{s\to\infty}\|\mathcal{R}(\cdot,s)\|_{H^1_\rho}=0$, which rewrites in original variables as
$$	\frac{u(x,T-\sigma)}{G^{-1}(\sigma)}\ge 1+\frac{2n}{4p|\log\sigma|}-\frac{|x|^2}{4p\sigma|\log\sigma|}
+\frac{\mathcal{R}_\sigma(x)}{|\log\sigma|},
\quad\hbox{where } \mathcal{R}_\sigma(x)=\mathcal{R}(x\sigma^{-1/2},|\log \sigma|).$$
Consequently,
$$\begin{aligned}
V(x,t)&=M-S(t-t_0)M+S(t-t_0)u(t_0)\ge S(t-t_0)u(t_0)\\
&\ge G^{-1}(\sigma)S(t-t_0)\Bigl(1+\frac{2n}{4p|\log\sigma|}-\frac{|x|^2}{4p\sigma|\log\sigma|}\Bigl)
+\frac{G^{-1}(\sigma)}{|\log\sigma|}S(t-t_0)\mathcal{R}_\sigma\equiv V_1+V_2\\
\end{aligned}$$
hence, since $V\ge 0$,
\be{VV1V2}
V\ge (V_1)_+-|V_2|\quad \hbox{ in $B_\delta\times(t_0,T)$.}
\ee
Using Lemma~\ref{lemSrho}(i) and recalling \eqref{deft0}, 
we may estimate $(V_1)_+$ from below for $t\in[t_0,T)$ as:
\be{V1below}
\begin{aligned}
(V_1(t))_+
&\ge e^{-\lambda(t-t_0)} G^{-1}(\sigma)\Bigl[1+\frac{2n}{4p|\log\sigma|}-\frac{|x|^2+2n(t-t_0)}{4p\sigma|\log\sigma|}\Bigl]_+
\chi_{B_{\delta/2}} \\
&\ge e^{-\lambda\sigma} G^{-1}(\sigma)\Bigl[1-\frac{|x|^2}{4p\sigma|\log\sigma|}\Bigl]_+
\chi_{B_{\delta/2}}. 
\end{aligned}
\ee
To control $V_2$, we observe that, for $t\in[t_0,T)$,
$$\int_{B_\delta}\mathcal{R}^2_\sigma(z)e^{-\frac{|z|^2}{4(t-t_0)}}dz
\le\int_{B_\delta} \mathcal{R}^2(z\sigma^{-1/2},|\log \sigma|)e^{-\frac{|z|^2}{4\sigma}}dz
\le\sigma^{n/2}\|\mathcal{R}(\cdot,|\log \sigma|)\|^2_{L^2_\rho}$$
and then use Lemma~\ref{lemSrho}(ii) 
to write
$$\begin{aligned}
\bigl|S(t-t_0)\mathcal{R}_\sigma\bigr|(x)
&\le C(n)(t-t_0)^{-n/4}\bigg(\int_{B_\delta}\mathcal{R}^2_\sigma(z)e^{-\frac{|z|^2}{4(t-t_0)}}dz\bigg)^{1/2}e^{\frac{|x|^2}{2(t-t_0)}}\\
&\le C(n)(t-t_0)^{-n/4}\sigma^{n/4}e^{\frac{|x|^2}{2(t-t_0)}}\|\mathcal{R}(\cdot,|\log \sigma|)\|_{L^2_\rho}.
\end{aligned}$$
Denoting $D_\sigma=\overline B_{\sqrt\sigma}\times [T-\textstyle\frac{\sigma}{2},T)$ and using \eqref{deft0},
we obtain
\be{controlV2}
\sup_{D_\sigma}|V_2|
\le C(n)\frac{G^{-1}(\sigma)}{|\log\sigma|}\|\mathcal{R}(\cdot,|\log \sigma|)\|_{L^2_\rho}
=\frac{G^{-1}(\sigma)\eps(\sigma)}{|\log\sigma|}.
\ee
Here and in the rest of the proof $\eps$ denotes a generic function such that $\lim_{\sigma\to 0}\eps(\sigma)=0$.
It follows from \eqref{VV1V2}-\eqref{controlV2} that 
\be{controlV}
V\ge e^{-\lambda\sigma} G^{-1}(\sigma)\Bigl(1-\frac{1+\eps(\sigma)}{4p|\log\sigma|}\Bigl)\quad\hbox{in $D_\sigma$}.
\ee
Moreover, we may choose $\sigma_0\in(0,\delta^2)$ such that the RHS of \eqref{controlV} is $>M$ 
for all $\sigma\in(0,\sigma_0]$.
Recalling \eqref{deftau1} and $\tau_1\ge T$, we deduce that, for all $\sigma\in(0,\sigma_0]$,
$$
0<G(V)+t_0-t\le  \mu(\sigma,t):=G\biggl[ e^{-\lambda\sigma} G^{-1}(\sigma)\Bigl(1-\frac{1+\eps(\sigma)}{4p|\log\sigma|}\Bigl)\biggr]+T-t-\sigma
\quad\hbox{in $D_\sigma$},
$$
where the RHS is $<G(M)$ hence, by \eqref{defUV}, \eqref{qz},
\be{fact}
u\ge G^{-1}(\mu(\sigma,t))\quad\hbox{in $D_\sigma$}.
\ee

\smallskip
{\bf Step 4.} {\it Asympotics of $\mu(\sigma,t)$ and conclusion.}
Fix any $x_0$ such that $|x_0|^2\le\sigma_0$. Choosing 
\be{choicesigma}
\sigma=|x_0|^2,
\ee
it follows from \eqref{fact} that
\be{fact2}
u(x_0, t)\ge G^{-1}(\mu(\sigma,t))\quad\hbox{ for $T-\frac{\sigma}{2}\le t<T$.}
\ee
To estimate $\mu$ from above, we first   use \eqref{ell12a2} in Lemma~\ref{ell12a}(i) to get, for some constant $C>0$,
$$
\mu(\sigma,t)\le (1+C\sigma)G\biggl[G^{-1}(\sigma)\Bigl(1-\frac{1+\eps(\sigma)}{4p|\log\sigma|}\Bigl)\biggr]+T-t-\sigma.
$$
Next, by the mean value theorem, there exists $\theta(\sigma)\in(0,1)$ such that
$$\mu(\sigma,t)\le (1+C\sigma)\biggl\{\sigma-
G'\Bigl[G^{-1}(\sigma)\Bigl(1-\theta(\sigma)\frac{1+\eps(\sigma)}{4p|\log\sigma|}\Bigl)\Bigl]
\frac{(1+\eps(\sigma))G^{-1}(\sigma)}{4p|\log\sigma|}\biggr\}+T-t-\sigma$$
and,  by \eqref{ell12a1} in Lemma~\ref{ell12a}(i) and $G'=-1/f$, we have
$$G'\Bigl[G^{-1}(\sigma)\Bigl(1-\theta(\sigma)\frac{1+\eps(\sigma)}{4p|\log\sigma|}\Bigl)\Bigl]
=\frac{-1}{f\bigl[G^{-1}(\sigma)\bigl(1-\theta(\sigma)\frac{1+\eps(\sigma)}{4p|\log\sigma|}\bigl)\bigl]}
\ge-\frac{1+\eps(\sigma)}{f\bigl(G^{-1}(\sigma))}.$$
On the other hand, by \eqref{ch} in Lemma~\ref{lm1}, we have
$\frac{X}{f(X)}=\frac{X^{1-p}}{L(X)}\sim (p-1)G(X)$
as $X\to\infty$. Using also \eqref{equivGG}, it follows that
$$\begin{aligned}
\mu(\sigma,t)
&\le (1+C\sigma)\biggl\{\sigma+\frac{(1+\eps(\sigma))G^{-1}(\sigma)}{4p|\log\sigma|f(G^{-1}(\sigma))}\biggr\}+T-t-\sigma
=T-t+C\sigma^2+ \frac{(1+\eps(\sigma))G^{-1}(\sigma)}{4p|\log\sigma|f(G^{-1}(\sigma))}\\
&\le T-t+C\sigma^2+\frac{p-1}{4p}\frac{(1+\eps(\sigma))\sigma}{|\log\sigma|}
\le(1+\eps(\sigma))\Bigl(T-t+\frac{p-1}{4p}\frac{\sigma}{|\log\sigma|}\Bigr).
\end{aligned}$$
By \eqref{choicesigma}, \eqref{fact2}, using \eqref{equivGG} again, we thus deduce that, as $(x,t)\to (0,T)$,
$$u(x,t)\ge(1+o(1))G^{-1}\bigg(T-t+\frac{p-1}{8p}\frac{|x|^2}{|\log|x||}\bigg),
\quad\hbox{ uniformly for $ |x|^2\ge 2(T-t)$.}$$
But, on the other hand, by \eqref{A1} and \eqref{equivGG}, we know that, as $(x,t)\to (0,T)$,
$$u(x,t)= (1+o(1))G^{-1}(T-t)=(1+o(1))G^{-1}\bigg(T-t+\frac{p-1}{8p}\frac{|x|^2}{|\log|x||}\bigg),
\ \hbox{ uniformly for $ |x|^2\le 2(T-t)$.}$$
The lower part of \eqref{infer1} follows.
\end{proof}

\begin{proof}[Proof of Corollary~\ref{Debo}]
(i) Letting $t\to T$ in the RHS of \eqref{infer1} and noting that $G^{-1}$ is continuous, we have 
$$
u(x,T):=\lim_{t\to T} u(x,t)= (1+o(1))G^{-1}\Bigl(\frac{(p-1)|x|^2}{8p|\log|x||}\Bigr) \quad \text{as } x\to 0.
$$

(ii) Let $K>0$.
In view of \eqref{infer1} and \eqref{equivGG}, setting $\xi=x/\sqrt{(T-t)|\log(T-t)|}$, it suffices to show that
\be{limdeltaK}
\lim_{t\to T} \Bigl\{\sup_{0<|\xi|\le K} |\delta(x,t)|\Bigr\}=0,
\quad\hbox{where }\delta(x,t):=\displaystyle\frac{T-t+\frac{p-1}{4p}\frac{|x|^2}{|\log|x|^2|}}{(T-t)(1+\frac{p-1}{4p}{|\xi|^2})}-1.
\ee
To this end observe that 
$$|\delta(x,t)|=\frac{(p-1){|\xi|^2}}{4p}\frac{\bigl|\frac{|\log(T-t)|}{|\log|x|^2|}-1\bigr|}{1+\frac{p-1}{4p}{|\xi|^2}}
\le \biggl|\frac{|\log(T-t)|}{|\log|x|^2|}-1\biggr|{|\xi|^2}.$$
If $0<{|\xi|^2}|\log(T-t)|\le 1$ then 
$|\log(T-t)|\le |\log|x|^2|$ for $t$ close to $T$, hence
$|\delta(x,t)|\le |\xi|^2\le \frac{1}{|\log(T-t)|}$.
Whereas if ${|\xi|^2}|\log(T-t)|\ge 1$ and $|\xi|\le K$ then, for all $\eps>0$, we have
$(1+\eps)^{-1}|\log(T-t)|\le |\log|x|^2|\le |\log(T-t)|$ for $t$ close to $T$,
hence $|\delta(x,t)|\le K^2\eps$. Property \eqref{limdeltaK}, hence \eqref{lou}, follows.
\end{proof}

\section{Proof of Theorem~\ref{Thm1}}\label{sec} 

The proof follows the line of
 arguments in \cite{souplet2019simplified} for the case $f(u)=u^p$,
which was a significant simplification (in the radial decreasing framework) of the original proofs from 
 \cite{filippas1992refined, HV93, Vel93b}.
However, some nontrivial difficulties arise in the case of non scale invariant nonlinearities
(due in particular to the new factors $h$ and $L$ in equation \eqref{(5.2)} below).
The proof is thus rather long and technical and we split it in several steps for clarity.

\subsection{Preliminaries}

Since the functional and spectral framework used in the proof of Theorem~\ref{Thm1} requires to work in the whole space,
we first need to extend the solution to $\R^n$ by means of a suitable cut-off in order to handle the case $\Omega=B_R$
(actually, the cut-off procedure can be applied even in the case $\Omega=\R^n$, so as to take advantage of the support 
and regularity properties of the solution).
\smallskip

Namely, for $\delta$ given by \eqref{uparabbdry},
we introduce a cut-off function $\phi\in C^\infty(\R^n)$, radially symmetric and nonincreasing in $|x|$, such that 
\be{defcutoff}
	\phi(x)=\begin{cases}
	1,\quad 0\le |x|\le\delta/2\\
	\noalign{\vskip 1mm}
	0,\quad |x|\ge \delta
	\end{cases}
\ee
	and we then set $\tilde u(x,t)=\phi(x)u(x,t)$. We see that $\tilde u$ is solution of
$$
	\tilde u_t-\Delta \tilde u= f(\tilde u)+A(x,t),\quad x\in \mathbb{R}^n,\ 0<t<T
$$
	where 
	$$A(x,t)=\phi f(u)-f(\phi u)-u\Delta\phi -2\nabla u\cdot\nabla \phi.$$ 
	Note that 
	\be{suppA}
	 \text{supp}(A)\subset\bigl\{(x,t)\in \mathbb{R}^n\times(0,T);\ \delta/2\le|x|\le\delta\bigr\}
	\quad\hbox{ and }\quad 	|A|+|\nabla A|\le C,
	\ee
	since $0$ is the only blow-up point of $u$.
Also, by \eqref{fposM}, \eqref{uparabbdry} and parabolic regularity we see that 
	\be{higher-regul}
	\nabla\tilde u\in C^{2,1}(\R^n\times(T-\delta,T)).
	\ee
	
Following \cite{duong2018construction} (where the case of a logarithmic nonlinearity was considered)
and \cite{chabi1} (for the general case), 
we introduce $w$, the rescaled solution by similarity variables and ODE normalization around $(0,T)$:
\be{defw}
\tilde u(x,t)=\psi(t)w(y,s),\quad y:=\frac{x}{\sqrt{T-t}}\quad s:= -\log(T-t).
\ee
Let $s_0:=-\log T$. In the rest of the proof, we will denote by $s_1$ a (sufficiently large) time $>\max(s_0,1)$,
which may vary from line to line.
 By direct calculation (see \cite{chabi1} for details), we see that $w$ is global solution of
\begin{equation}\label{(5.2)}
w_s+\mathcal{L}w=Z(w,s) + B(y,s), \quad y\in \mathbb{R}^n,\ s\in (s_0,\infty),
\end{equation}
where 
$$\mathcal{L}w=-\Delta w+\frac{1}{2}y\cdot\nabla w$$
is the Hermite operator, and
\be{LO1a}
Z(\xi,s):=h(s)\Bigl(\xi^p\frac{L(\psi_1\xi)}{L(\psi_1)}-\xi\Bigr),
\ee
\be{LO1}
h(s):=e^{-s}\psi_1^{p-1}(s)L(\psi_1(s)), \quad \psi_1(s):=\psi(T-e^{-s}),
\quad 	B(y,s)=\frac{A(ye^{-s/2},T-e^{-s})}{e^s\psi_1(s)}.
\ee
 By \eqref{type1b} and \eqref{jojo0} in Theorem~\ref{RDT} there exists $M_0>1$ such that 
\begin{equation}\label{(1.9)}
0\le w\le M_0,\quad |\nabla w|\le M_0 \quad\hbox{ and }\quad 
\lim_{s\to \infty} w(y,s)=1\ \text{uniformly for } y \text{ bounded,} 
\end{equation}
and, moreover, by \eqref{suppA},
\be{suppB}
	 \text{supp}(B)\subset\bigl\{(y,s)\in\mathbb{R}^n\times (s_0,\infty): \textstyle\frac{\delta}{2} e^{s/2}\le |y|\le \delta e^{s/2} \bigr\}
	\quad\hbox{ and }\quad 	|B|+|\nabla B|\le Ce^{-s}.
	\ee
	
We claim that
\be{limhbeta}
\lim_{s\to\infty}h(s)=\beta
\ee
and that
\be{boundW}
\sup_{s>s_1,\ y\in\R^n} \bigl(|Z(w(y,s),s)|+ |\partial_\xi  Z(w(y,s),s)|\bigr)<\infty.
\ee
Indeed, it follows from \eqref{ch} that
		$e^{-s}=G(\psi_1(s))\sim\beta\frac{\psi_1^{1-p}(s)}{L(\psi_1(s))}=\beta\frac{e^{-s}}{h(s)}$ as $s\to\infty$,
	hence~\eqref{limhbeta}. 
To check \eqref{boundW}, we first note that \eqref{sem0}-\eqref{sem} guarantees the existence of $X_0$ such that $p/2\le \frac{Xf'(X)}{f(X)}\le 2p$ and $\tilde f(X):=\frac{f(X)}{X}$ is increasing on $[X_0,\infty)$,
and we write
\be{Wxis}
Z(\xi,s)=h(s)\Bigl(\frac{f(\psi_1\xi)}{f(\psi_1)}-\xi\Bigr),\quad
\partial_\xi Z(\xi,s)=-h(s)\Bigl(\frac{\psi_1f'(\psi_1\xi)}{f(\psi_1)}-1\Bigr),
\ee
where we denote $\psi_1=\psi_1(s)$ for simplicity.
Let $\xi\in [0,M_0]$. If $\psi_1\xi\ge X_0$, then 
$0\le \frac{f(\psi_1\xi)}{f(\psi_1)}\le \frac{f(M_0\psi_1)}{f(\psi_1)}\le C$ and
$0\le \frac{\psi_1f'(\psi_1\xi)}{f(\psi_1)}\le \frac{2pf(\psi_1\xi)}{\xi f(\psi_1)}=\frac{2p\tilde f(\psi_1\xi)}{\tilde f(\psi_1)}\le \frac{2p\tilde f(M_0\psi_1)}{\tilde f(\psi_1)}\le C$, owing to the fact that $L$ has slow variation at infinity.
Whereas, if $\psi_1\xi\le X_0$, then 
$|\frac{f(\psi_1\xi)}{f(\psi_1)}|\le \frac{C}{f(\psi_1)}\le C$ and
$|\frac{\psi_1f'(\psi_1\xi)}{f(\psi_1)}|\le C\frac{\psi_1}{f(\psi_1)}\le C$. 
In view of \eqref{limhbeta} and \eqref{Wxis}, this proves~\eqref{boundW}.

\smallskip

We also record the following easy bounds for $\psi$:
		\be{compfX}
\frac{c_1}{T-t}\le \frac{f(\psi(t))}{\psi(t)}  \le \frac{c_2}{T-t},\quad t\to T,
\ee
		\be{compfX0}
		(T-t)^{-\beta+\eps}\le \psi(t)\le (T-t)^{-\beta-\eps}, \quad t\to T,
		\ee
		for all $\eps>0$ and some constants $c_1, c_2>0$.
		Estimate \eqref{compfX} follows from the fact that
		$T-t=G(\psi(t))\sim\beta\frac{\psi(t)}{f(\psi(t))}$ as $t\to T$
		(cf.~\eqref{ch} in Lemma~\ref{lm1}).
Then, since $L$ has slow variation, it is easily checked that, for all $\eps>0$,
$X^{p-1-\eps}\le f(X)/X\le X^{p-1+\eps}$ as $X\to\infty$, and \eqref{compfX0} follows from~\eqref{compfX}.

\smallskip

 Now, recalling definition \eqref{defL2H1}, we respectively denote by 
$(v,w):=\int_{\mathbb{R}^n}vw\rho dy$ and $\|v\|=(v,v)^{\frac{1}{2}}$ the inner product and the norm of the Hilbert space $L^2_\rho$.
Also, $\perp$ will denote the orthogonality in the $L^2_\rho$ sense.
We set 
$$\varphi:=1-w.$$
For each $R > 0$, we have
\begin{equation}\label{(5.13)} m_R(s):= \underset{| y | \leq R}\sup |\varphi(y,s)| \to 0\quad \text{as}\ \ s\to \infty,\end{equation}
and 
	\be{varphi0}
\underset{s\to\infty}{\lim}\| \varphi(\cdot,s) \| = 0,
\ee
as a consequence of \eqref{(1.9)} and dominated convergence. 
In terms of $\varphi$, since $u$ blows up only at the origin and recalling \eqref{defcutoff}, \eqref{defw}
and \eqref{compfX0},
the desired estimate \eqref{A1} is equivalent to 
\begin{equation}\label{(5.4)}
\varphi(y,s)=\frac{|y|^2-2n}{4ps}+o\Bigl(\frac{1}{s}\Bigl)\quad \text{as } s\to \infty,
\end{equation}
with convergence in $H^1_\rho(\mathbb{R}^n)$ and uniformly for $y$ bounded.
The idea of the proof of \eqref{(5.4)} is to linearize \eqref{(5.2)} around 
$w=1$, i.e., $\varphi=0$, and perform a kind of center manifold analysis.

\smallskip
We shall use the following weighted Poincar\'e and  Poincar\'e-Wirtinger-type inequalities 
 (see, e.g.,~\cite{souplet2019simplified} for a proof).

\begin{prop}\label{prop1}
	We have 
	$$
	\int_{\mathbb{R}^n}|y|^2v^2\rho dy\le 16\| \nabla v\|^2+4n\| v\|^2,\quad\hbox{for all }v\in H^1_\rho.
$$
\end{prop}

\begin{prop}\label{prop2}
	Let  $ v \in  H^{1}_{\rho}$ 
	\begin{itemize}
			\item [(i)] If $(v, y_j) = 0$ then, for all  $j \in  \{ 1, . . . , n\}$, we have
	\begin{equation}\label{1}
	\| v\|^2 \leq \| \nabla v \|^{2} + \bar{v}^{2} \ where \ \bar{v} = \left( \int_{\mathbb{R}^{n}} \rho dy\right)^{-1/2} \int_{\mathbb{R}^{n}}v \rho dy.
	\end{equation} 
	
	In particular, \eqref{1}  is true whenever $v$ is radially symmetric. 
		\smallskip
		
	\item [(ii)]   If $v$ is orthogonal to  all\  polynomials of degree $\leq 3$,  then  we  have
	\begin{equation}\label{2}
	\| v\|^2 \leq \frac{1}{2}\| \nabla v \|^{2}.
	\end{equation} 
	
	In particular, \eqref{2} is true whenever $v$ is radially symmetric and $(v,1) = (v,|y|^{2}) = 0$.
	\smallskip
	
	\item [(iii)]   Let $ i \in \{1,...,n\}
	$ and  assume $ \partial_{y_i}v \in H^{1}_{\rho}$.  If $(v,y_i) = (v, y_i^{2}-2) = 0 \ and\ (v,y_iy_j) = 0 \ for \ all \ j\neq i$, \ then 
	\begin{equation}\label{3}
	\| \partial_{y_j}v \|^{2} \leq \| \nabla(\partial_{y_i}v) \|^{2}.
	\end{equation}
	
	In particular, \eqref{3} is true for all $i \in \{1,...,n\}$ whenever $v$ is radially symmetric, $\nabla v \in H^{1}_{\rho}$
	 and $(v,|y|^{2} -2n) = 0.$
	 
	 \end{itemize}
\end{prop}

We shall also need the following sharp expansion of the function $G$, which is an improved version of \eqref{ch} of Lemma~\ref{lm1} and whose proof is postponed to  the Appendix.

\begin{lem}\label{lm1b} 
Under assumptions \eqref{sem0}-\eqref{sem}, we have
$$G(X)\Bigl((p-1)X^{p-1}L(X)+X^pL'(X)\Bigr)=1+o\Bigl(\frac{1}{\log X}\Bigr),\quad X\to\infty.$$
 \end{lem}

\subsection{Proof of nonexponential decay of $\varphi$.} 

The following lemma provides a polynomial lower bound on the decay of $\varphi$ in $L^{2}_{\rho}$, 
 which is 
an important piece of information in the proof of Theorem~\ref{Thm1}. 
As in \cite{souplet2019simplified}, this lower bound is obtained as a 
consequence of the maximum principle argument leading to the upper part of \eqref{infer1}, which we have already proved in section~\ref{upperrrrrr}. 
\begin{lem}\label{lem}
	Under the assumptions of Theorem~\ref{theo}, there exists $c > 0$  
	such that
	\be{lowerdecayw}
	\|\varphi(.,s)\| \geq cs^{-1}, \quad  s\ge s_1.
	\ee
	
\end{lem}

\begin{proof}
	
Fix $K > 0$. By \eqref{(1.9)} there exists $t_0(K) \in (T/2, T)$ such that 
$$
	u(r,t) \geq \frac{\psi(t)}{2}\ge M, \quad 0\leq r \leq K\sqrt{T-t}, \ \ t_0 < t < T,
$$
	where $M$ is as in Lemma~\ref{lmfF}. Then  $u= u(r,t),\  r=|x|,$ satisfies (cf.~\eqref{PJcontrad0})
	\begin{equation}\label{(5.10)}
-u_r(r,t) \geq \frac{r}{2} \frac{f(u)}{A +\log f(u) },\quad  0 \leq r \leq K\sqrt{T-t}, \ \ t_0 < t < T.
\end{equation}
Let $\varphi =1-w,$ with $w =w (\eta, s)$ and  $\eta =  | y |$. Observing that the RHS of \eqref{(5.10)} is an increasing function of $u$ for large $u$ and using \eqref{prop}, we deduce that, for
$s_2(K)\ge \max(s_0,1)$ sufficiently large 
and some $c_1>0$,
\be{roi}
\begin{aligned}
 \varphi_\eta(\eta, s) &= -w_\eta (\eta, s) = -\frac{\sqrt{T-t}}{\psi(t)} \, u_r(\eta \sqrt{T-t},t) \\
 &\ge
 \frac{ \eta(T-t)f(\frac{\psi(t)}{2})}{2\psi(t)\bigl(A + \log f(\frac{\psi(t)}{2})\bigr)} \ge 
\frac{c_1 \eta(T-t)f(\psi(t))}{\psi(t)\bigl(A + \log f(\psi(t))\bigr)},\quad\hbox{ for all $0 \leq \eta \leq K$ and $s \ge s_2$}.
\end{aligned}
\ee
Inequalities \eqref{compfX}, \eqref{compfX0} guarantee that
$\log(f(\psi(t)))\le C|\log(T-t)|$ as $t\to T$.
This combined with \eqref{compfX}, \eqref{roi} implies
 \begin{equation*}
 \varphi_\eta(\eta, s)\ge \frac{c\eta}{s}\quad\hbox{for all $0 \leq \eta \leq K$ and $s \ge s_2$},
 \end{equation*}
 where the constant $c > 0$ is independent of $K$.

Now, choose $K = 2(1 + c^{-1})$ and take any $s \ge s_2(K).$ If $ \varphi(1,s) \geq -1/s$ then it follows that 
\begin{equation*} \varphi(\eta, s) = \varphi(1,s ) + \int_{1}^{\eta} \varphi_\eta(z,s)dz \geq \frac{-1 + c(\eta -1)}{s} \geq \frac{1}{s}, \ \ K-1 \leq \eta \leq K, \end{equation*}
and hence $\|\varphi(\cdot,s)\| \geq \bigl(\int_{K-1\leq | y | \leq K} \rho\bigr)^{1/2}s^{-1}$. Otherwise, we have $\varphi(1,s) \leq -1/s$ and, since $\varphi$ is a nondecreasing function of $\eta$, we get $\varphi(\eta, s) \leq -1/s$ for $\eta \in [0,1];$ hence $\|\varphi (\cdot,s)\| \geq \bigl(\int_{| y | \leq 1} \rho\bigr)^{1/2}s^{-1}.$ We conclude that $\|\varphi (\cdot,s)\| \geq cs^{-1}$ for all $s \geq s_2.$
\end{proof}

\subsection{ Proof of Theorem~\ref{Thm1}.}

\begin{proof}  
Integrals over $\mathbb{R}^n$ will be simply denoted by $\int$, and the variables will be omitted when no confusion arises. For clarity we split the proof into several steps.$\\$ 

\textit{\textbf{Step 1.} Bounds and equation for $\varphi$}.
Note that under the current assumptions,
 $\varphi$ is radially symmetric nondecreasing, and
\be{boundvarphi}
-M_0 \le \varphi \le 1
\ee
 as a consequence of \eqref{(1.9)}. 
We rewrite equation \eqref{(5.2)} in terms of $\varphi$ as 
 \be{(5.11)} 
 \varphi_s + \mathcal{L}\varphi-\varphi =W(\varphi,s) - B(y,s)\equiv\tau(s)\varphi+F(\varphi,s)- B(y,s),
 \quad
y\in\R^n,\ s\ge s_1,
\ee
with
\be{defWT}
W(\xi,s):=-Z(1-\xi,s) -\xi, \quad \tau(s):=\partial_\xi W(0,s),\quad F(\xi,s)=W(\xi,s)-\tau(s)\xi.
\ee

Whereas $W(\xi,s)$ in the case corresponding to $f(u)=u^p$ was bounded by a quadratic function 
(cf.~\cite{souplet2019simplified}), the function $W$ here has a nonzero linear part, but 
its decay as $s\to\infty$ will turn out to be sufficiently fast
so as to permit the required linearization.
Namely, we claim that 
\be{W0s}
\tau(s)=o\Bigl(\frac{1}{s}\Bigr),\quad s\to\infty,
\ee
\be{(5.14)} 
|F(\varphi,s)| \le C\varphi^2,  \quad
|\partial_\xi F(\varphi,s)|\le C|\varphi|,\quad y\in\R^n,\ s\ge s_1
\ee
and 
\be{cvFV}
\lim_{\xi\to 0}\frac{F(\xi,s)}{\xi^2}=-h(s)\Bigl(\frac{p(p-1)}{2}+p\psi_1\frac{L'(\psi_1)}{L(\psi_1)}
+\frac{\psi_1^2}{2}\frac{L''(\psi_1)}{L(\psi_1)}\Bigr),\quad  s\ge s_1.
\ee

To this end, we write
$W(\xi,s)=-h(s)W_1(\xi,s)-\xi$ 
with 
$W_1(\xi,s) = (1-\xi)^p\frac{L((1-\xi)\psi_1)}{L(\psi_1)}+\xi-1$,
and we compute
$$\partial_\xi W_1= 1-p(1-\xi)^{p-1}\frac{L((1-\xi)\psi_1)}{L(\psi_1)}
-(1-\xi)^p\psi_1\frac{L'((1-\xi)\psi_1)}{L(\psi_1)}.$$
On the other hand,
we have $\psi(t)\ge C(T-t)^{-1/p}$ as $t\to T$, owing to \eqref{compfX0}, hence
$\log\psi_1(s)\ge Cs$ as $s\to\infty$.
Recalling $h(s)=e^{-s}\psi_1^{p-1}(s)L(\psi_1(s))$ with $\psi_1(s):=\psi(T-e^{-s})=G^{-1}(e^{-s})$, 
and using the expansion in Lemma~\ref{lm1b}, we get
$$\begin{aligned}
\partial_\xi W(0,s)
&=-h(s)\Bigl(1-p-\psi_1\frac{L'(\psi_1)}{L(\psi_1)}\Bigr)-1
=e^{-s}\bigl((p-1)\psi_1^{p-1}L(\psi_1)+\psi_1^p L'(\psi_1)\bigr)-1 \\
&=G(X)\bigl((p-1)X^{p-1}L(X)+X^pL'(X)\bigr)_{|X=\psi_1(s)}-1=o\Bigl(\frac{1}{\log\psi_1(s)}\Bigr)
=o\Bigl(\frac{1}{s}\Bigr),
\end{aligned}$$
which proves claim \eqref{W0s}.

To check \eqref{(5.14)}, since $L$ is $C^2$ for $X$ large and $W(0,s)=0$, 
we may 
use Taylor's formula with integral remainder to write,
for $s\ge s_1$ and $-M_0\le \xi\le 1/2$,
\be{TaylorRI} 
F(\xi,s)=W(\xi,s)-\partial_\xi W(0,s)\xi=
\xi^2\int^{1}_{0} \partial^2_\xi W(t\xi,s)(1-t)dt,
\ee
\be{TaylorRI2}
\partial_\xi F(\xi,s)=\partial_\xi W(\xi,s)-\partial_\xi W(0,s)
=\xi\int_0^1 \partial^2_\xi W(t\xi,s)dt,
\ee
where
\be{TaylorRI3}
\begin{aligned}
\partial^2_\xi W_1(\xi,s)
&=p(p-1)(1-\xi)^{p-2}\frac{L((1-\xi)\psi_1)}{L(\psi_1)}\\
&\quad +2p(1-\xi)^{p-1}\psi_1\frac{L'((1-\xi)\psi_1)}{L(\psi_1)}
+(1-\xi)^p\psi_1^2\frac{L''((1-\xi)\psi_1)}{L(\psi_1)},\quad -M_0\le \xi\le 1/2,
\end{aligned}
\ee
satisfies the bound
$$\begin{aligned}
|\partial^2_\xi W_1(\xi,s)|
&\le C\Bigl\{\frac{L((1-\xi)\psi_1)}{L(\psi_1)}+\psi_1\frac{|L'((1-\xi)\psi_1)|}{L(\psi_1)}
+\psi_1^2\frac{|L''((1-\xi)\psi_1)|}{L(\psi_1)}\Bigr\} \\
&\le C\frac{L((1-\xi)\psi_1)}{L(\psi_1)}\Bigl\{1+\psi_1\frac{|L'((1-\xi)\psi_1)|}{L((1-\xi)\psi_1)}
+\psi_1^2\frac{|L''((1-\xi)\psi_1)|}{L((1-\xi)\psi_1)}\Bigr\}\le C,\quad -M_0\le \xi\le 1/2,
\end{aligned}$$
 owing to \eqref{limLprime}.
This, combined with \eqref{TaylorRI}-\eqref{TaylorRI2}, yields \eqref{(5.14)} at points $(y,s)$ such that $\varphi(y,s)\le 1/2$. Estimate \eqref{(5.14)} at points such that $\varphi(y,s)>1/2$ directly follows from \eqref{boundW} and \eqref{defWT}. As for \eqref{cvFV} it is a consequence of \eqref{TaylorRI} and \eqref{TaylorRI3}.

\smallskip

\textit{\textbf{Step 2.} Decomposition of $\varphi$}.
It is well known (see, e.g., \cite{filippas1992refined,duong2018construction,souplet2019simplified}) that $\mathcal{L}_0:=\Delta-\frac{1}{2}y\cdot\nabla+1$ acting on $L^2_\rho$ and restricted to symmetric functions, has:
\begin{itemize}
	\item [-] one unstable direction, corresponding to constant eigenfunctions;
	\item [-] one neutral direction, colinear to the quadratic eigenfunction $|y|^2-2n$; and 
	\item [-] a stable subspace of codimension two.
\end{itemize} 
Now set $H_0 = c_0$ and $H_2 = c_2P$ with $P(y) = |y|^2 - 2n$. We may choose the normalization constants $c_0, c_2 > 0$ so that $\| H_0 \| = \| H_2 \| = 1$ and $H_0 \perp H_2$. We then define the orthogonal decomposition of $\varphi$ into ``unstable'', ``neutral'' and ``stable'' components as follows: 
\begin{equation}\label{(5.16)}  \varphi = a(s)H_0 + b(s)H_2(y) + \theta(y,s), \end{equation}
where $a(s):=(\varphi(\cdot,s), H_0)$, $b(s):=(\varphi(\cdot,s), H_2)$ and $\theta := \varphi - a(s)H_0 -b(s)H_2(y).$

Consequently, we have $\theta(\cdot,s) \perp H_0$, $\theta(\cdot,s) \perp H_2.$ Substituting the decomposition \eqref{(5.16)} in the PDE \eqref{(5.11)} and using $\mathcal{L}H_i=\frac{i}{2}H_i$ for $i=0,2$, we get
\begin{equation} \label{(5.17)}  
a'(s)H_0 + b'(s)H_2(y) + \theta_s + \mathcal{L}\theta = a(s)H_0 + \theta
 +\tau(s)\varphi +F(\varphi,s)- B(y,s).
 \end{equation}
Integrating by parts, we obtain $(\mathcal{L}\theta, H_i) = (\theta, \mathcal{L}H_i)=(i/2)(\theta, H_i) = 0$ for $i=0, 2.$ 
 All integrations are justified by the bounds \eqref{(1.9)}, \eqref{suppB} and \eqref{(5.14)} and the exponential decay of the weight $\rho$.
Taking scalar products and using the orthogonality relations, it follows that 
\begin{equation} \label{(5.18)} 
\begin{aligned}
a'(s) &= a(s) +\tau(s)a(s)+c_0 \int F(\varphi,s)\rho -c_0\int B(y,s)\rho, \\
b'(s) &=\tau(s)b(s) +\int F(\varphi,s)H_2\rho - \int B(y,s)H_2\rho.
\end{aligned}
\end{equation}
In what follows, we will denote by $\eps(s)$ various functions such that $\lim_{s\to \infty}\eps(s) = 0$,
and which may vary from line to line.

\smallskip

\textit{\textbf{Step 3.} Control of the unstable mode in} $L^{2}_{\rho}$. We shall show  that 
\begin{equation} \label{(5.19)} | a(s) | = o (\| \varphi (s) \|), \ \ s\to \infty.\end{equation}

Set $J(s) = \int \varphi^{2}\rho$ and $K(s) = \int \| \nabla \varphi \|^{2} \rho.$ The idea is to derive a simple differential inequality for the quantity $a^{2}-\lambda J$.
Fix any $\lambda \in (0, 1/4)$. As a consequence of the weighted Poincaré inequality for radial functions in Proposition~\ref{prop2}(i) we first have the relation
\begin{equation} \label{(5.20)} J \leq a^{2} + K.\end{equation}
Testing \eqref{(5.11)} with $\rho \varphi$, we obtain $\frac{1}{2}J'(s) = -K + (1+\tau(s))J + \int F(\varphi)\varphi \rho
- \int B(y,s)\varphi\rho$; hence, in view of \eqref{(5.18)}, 
\begin{equation} \label{(5.21)} 
\begin{aligned}
\frac{1}{2}(a^{2} - \lambda J)' = a^{2} &+ \lambda (K- J)+\tau(s)(a^2-\lambda J)\\
& + c_0 a(s) \int F(\varphi,s) \rho - \lambda \int F(\varphi,s)\varphi \rho
\ -c_0a(s)\int B(y,s)\rho+\lambda\int B(y,s)\varphi\rho.
\end{aligned}
\end{equation}
We proceed to show that the (nonlinear) integral terms in \eqref{(5.21)} are of lower order as $s \to \infty.$ First note that for large $s$ 
\begin{equation}\label{(5.22)} \left| \int F(\varphi,s)\varphi \rho \right| + \left| a(s) \int F(\varphi,s)\rho \right| 
\leq C \int | \varphi |^{3} \rho\end{equation} 
by \eqref{(5.14)}, the boundedness of $\varphi$ 
and Holder's inequality. To estimate $\int | \varphi |^{3} \rho $, we  apply the weighted Poincar\'e inequality in Proposition~\ref{prop1}, 
\begin{equation}\label{(5. 23)} \int \varphi^{2} |y|^{2} \rho \leq C \int (\varphi^{2} + |\nabla \varphi|^{2}) \rho,\end{equation}
along with the boundedness of $\varphi$, to write 
\begin{equation*}\int | \varphi |^{3} \rho = \int_{| y | \leq R} | \varphi |^{3} \rho + \int_{| y | > R} | \varphi |^{3}\rho \leq m_R(s)J + \frac{C}{R^{2}} \int_{| y | > R} \varphi^{2} | y |^{2}\rho  \leq m_R(s)J + CR^{-2}(J + K).\end{equation*}
For any $\eta > 0$, first choosing $ R= \eta^{-1/2}$ and 
using \eqref{(5.13)}, we obtain $\int | \varphi |^{3}\rho \leq 2C\eta(J + K)$ for all sufficiently large $s$; hence 
\begin{equation}\label{(5.24)}  \int  | \varphi |^{3}\rho \leq \eps(s)(J +K).
\end{equation}
To control the terms involving $B$ in this and the subsequent steps, we note that, by \eqref{suppB}, for any $\mu>0$, we have
\be{boundsB}
\int \bigl(|\nabla B|^2+B^2+|B|\bigr)(y,s)(1+H_2^2) \rho\le C e^{-\mu s},\quad s\ge s_1.
\ee
Therefore,
$$\Bigl|c_0a(s)\int B(y,s)\rho\Bigr|+\lambda\Bigl|\int B(y,s)\varphi\rho\Bigr|
\le \frac14 a^2(s)+C\int (B^2(y,s)+|B(y,s)|) \rho\le \frac14 a^2(s)+Ce^{-2 s}.$$
Now by combining with 
\eqref{W0s}, 
\eqref{(5.20)}-\eqref{(5.22)} and \eqref{(5.24)}, we obtain
$$
\begin{aligned}
\frac{1}{2}(a^{2} - \lambda J)' 
&\geq (1+\tau(s))a^{2}  
+ (\lambda - \eps (s))K - (\lambda+\lambda\tau(s) + \eps(s))J  
 -\frac14 a^2(s)-Ce^{-2 s}\\
&\geq \Bigl(\frac34-\lambda + \eps(s)\Bigr)a^{2} + \eps (s)J 
-Ce^{-2 s} \geq \frac{1}{2} (a^{2} - \lambda J) -Ce^{-2 s}
\end{aligned}$$
for $s\ge s_1$ large, hence
$\bigl(\bigl(a^{2} - \lambda J\bigl)e^{-s}-Ce^{-3 s}\bigr)'\ge 0$.
By integrating and using the boundedness of $\varphi$, we obtain
$$ \bigl(a^{2} - \lambda J\bigl)(s)e^{-s}-Ce^{-3 s}
\le \bigl(a^{2} - \lambda J\bigl)(\sigma)e^{-\sigma}-Ce^{-3 \sigma}\le Ce^{-\sigma},\quad \sigma>s\ge s_1.$$
Letting $\sigma\to\infty$, it follows that
$a^{2}(s)\le \lambda J(s)+Ce^{-2s}$ for all $s\ge s_1$.
Since this is true for any $\lambda \in (0, 1/4)$ and since, on the other hand, 
\be{lowerJ}
J\geq cs^{-2}
\ee
by \eqref{lowerdecayw} in Lemma~\ref{lem}, it follows that $a^{2} = o(J)$ as $s \to \infty$, which is equivalent to \eqref{(5.19)}.
\smallskip

\textit{\textbf{Step 4.} Control of the stable component in} $L^{2}_{\rho}$. We shall show that
\begin{equation} \label{(5.25)}\| \theta (s) \| = o(|b(s)|), \quad s\to \infty.\end{equation}
We set $\Theta(s) = \int \theta^{2} \rho,\ M(s) = \int | \nabla \theta |^{2} \rho$ and the idea is to derive a simple differential inequality for the quantity $\Theta - \lambda b^{2}.$

As a consequence of orthogonality of the decomposition \eqref{(5.16)} in $L^{2}_{\rho}$, also taking into account $(\nabla H_2, \nabla \theta)$ $ = (\mathcal{L}H_2, \theta) = (H_2, \theta) = 0$, we have
\begin{equation} \label{(5.26)} J = a^{2} + b^{2} + \Theta, \ \ \ K=b^{2}\|\nabla H_2\|^2 
 + M.\end{equation}
Moreover, since $\theta$ is radial and $\theta \perp H_0,\ \theta \perp H_2$, we may apply the 
Poincar\'e inequality in Proposition~\ref{prop2} to get 
\begin{equation} \label{(5.27)} 
\Theta \leq \frac{1}{2}M.
\end{equation}
We now test \eqref{(5.17)} with $\theta \rho$. Using 
$H_0 \perp \theta(.,s), \ H_2 \perp \theta(.,s)$ and noting that $(\mathcal{L} \theta, \theta) = (\nabla \theta, \nabla \theta)$, 
we obtain $\frac{1}{2}\Theta'(s)= -M +(1+\tau(s)) \Theta + \int F(\varphi,s) \theta \rho-\int B\theta\rho$. Fixing any $\lambda > 0$, we deduce from \eqref{(5.18)} that 
\begin{equation} \label{(5.28)} 
\begin{aligned} 
\frac{1}{2}(\Theta - \lambda b^{2})'(s)
&= -M + \Theta +\tau(s)(\Theta-\lambda b^2)+ \int F(\varphi,s) \theta \rho \\
&\qquad - \lambda b(s)\int F(\varphi,s)H_2 \rho
 -\int B(y,s)\theta \rho+\lambda b(s)\int B(y,s) H_2\rho.
\end{aligned}
\end{equation}

As in Step 3, we wish to control the integral terms in \eqref{(5.28)}. To this end, for each $\eta > 0$, we write
\begin{align*} \left| \int F(\varphi,s) \theta \rho \right|  + \left|  b(s)\int F(\varphi,s)H_2 \rho \right| \leq& \eta \int \theta^{2} \rho + C_\eta \int \varphi^{4}\rho + C | b(s)| \int \varphi^{2} (| y |^{2} + 1)\rho \\
\leq& \eta \Theta + [C_\eta \eps(s) + C| b(s) |](J + K),\end{align*}
where we used \eqref{(5.14)}, the boundedness of $\varphi$, \eqref{(5.24)}, \eqref{(5. 23)}. Also, owing to \eqref{(5.19)}, \eqref{(5.20)}, \eqref{(5.26)} and \eqref{(5.27)}, 
we observe that 
\be{JKL}
J+ K + \Theta \leq C(b^{2} + M).
\ee
Since $\lim_{s\to \infty} b(s) = 0$
owing to \eqref{varphi0}, 
we deduce that
\begin{equation} \label{(5.29)} \left| \int F(\varphi,s)\theta \rho \right|  + \left|  b(s)\int F(\varphi,s)H_2 \rho \right| \leq \eps(s)(b^{2} +M).\end{equation}
To control the terms involving $B$, we use \eqref{boundsB} to write
\begin{equation}\label{B_1}
\begin{aligned} 
\left| \int B(y,s) \theta \rho \right|  +\lambda \left|  b(s)\int B(y,s)H_2 \rho \right| 
&\le \frac18 \bigl(\Theta(s)+\lambda b^2(s)\bigr)+C\int B^2(y,s)(1+H_2^2)\rho \\
&\le \frac18 \bigl(\Theta(s)+\lambda b^2(s)\bigr)+Ce^{-2s}.
\end{aligned} 
\end{equation}
Now, this along with \eqref{(5.27)}-\eqref{(5.29)} 
 guarantees that 
$$\begin{aligned} 
\frac{1}{2}(\Theta - \lambda b^{2})'
& \leq -M + \Theta + \eps (s)(M + b^{2})+\tau(s)(\Theta-\lambda b^2) +\frac18 \bigl(\Theta+\lambda b^2\bigr)+Ce^{-2s}
 \le -\frac12\bigl(\Theta -\lambda b^2\bigl)+Ce^{-2s}.
\end{aligned} $$
Therefore, $(\Theta - \lambda b^{2})' \leq -(\Theta - \lambda b^{2})+Ce^{-2s}$ for $s$ large
and we easily deduce that $\Theta \leq \lambda b^{2} + Ce^{-s}$ as $s \to \infty$. 
But since, on the other hand, \eqref{lowerJ} along with \eqref{(5.19)}
and \eqref{(5.26)} guarantees that 
\be{compb2}
b^{2} \geq cs^{-2},
\ee
 we obtain \eqref{(5.25)}.

\smallskip

\textit{\textbf{Step 5.} Control of the stable component in} $H^{1}_{\rho}$. We shall show that 
\begin{equation}
\label{(5.30)} 
\| \nabla\theta(s) \| = o(|b(s)|) \ \ as\ s\to \infty.
\end{equation}

We proceed similary as for Step 4, this time working at the level of the equation satisfied by $\partial_{y_i}\varphi$. We will derive a differential inequality for $M - \lambda b^{2}$.
By \eqref{higher-regul} we have $D\varphi\in C^{2,1}(\R^n\times(s_1,\infty))$.
Fix any $i\in \{1,...,n\}$. Differentiating \eqref{(5.11)} 
and setting $F_i(\varphi, \varphi_i,s) :=\partial_{\xi} F(\varphi,s)\varphi_i$ 
and $B_i(y,s)=\partial_{y_i} B(y,s)$,
we see that $\varphi_i := \partial_{y_i}\varphi$ 
satisfies
\be{eqphii}
\partial_s \varphi_i + \mathcal{L}\varphi_i - \frac{1}{2}\varphi_i -\tau(s)\varphi_i
=F_i(\varphi, \varphi_i,s)-B_i(y,s),\quad y\in\R^n,\ s\ge s_1.
\ee
Moreover, $\varphi_i(\cdot,s)$ is supported in $\{|y|\le 2\delta e^{s/2}\}$ owing to \eqref{defcutoff}, \eqref{defw},
 and the RHS of \eqref{eqphii} is uniformly bounded in view of \eqref{(1.9)}, \eqref{suppB} and \eqref{(5.14)}.
Differentiating the decomposition in \eqref{(5.16)}, we get $\varphi_i = 2c_2b(s)y_i + \theta_i(y,s)$, where $\theta_i = \partial_{y_i}\theta$. Substituting in the last equation and using $\mathcal{L}y_i = \frac{1}{2}y_i$, we obtain
\begin{equation} \label{(5.31)} 2c_2b'(s)y_i + \partial_s \theta_i + \mathcal{L}\theta_i = \frac{1}{2}\theta_i + F_i(\varphi, \varphi_i,s)+\tau(s)\varphi_i -B_i(y,s).\end{equation}
Since $\theta \perp H_2$, it follows from the weighted Poincaré inequality in Proposition~\ref{prop2}(iii) 
that 
\begin{equation} \label{(5.32)}  M=\| \nabla\theta \|^{2} \leq N := \underset{i}{\sum}\| \nabla\theta_i \|^{2}.\end{equation}
Now fix $\lambda > 0$. 
By the properties after \eqref{eqphii}, we may test
\eqref{(5.31)} with $\theta_i \rho$. Summing over i, and using \eqref{(5.18)} and $(2c_2y, \nabla \theta) = (\nabla H_2, \nabla \theta) = 0$, we get 
\begin{equation}\label{(5.33)} 
\begin{aligned}
\frac{1}{2}(M - \lambda b^{2})' 
&= - N + \frac{1}{2}M + \underset{i}{\sum} \int F_i(\varphi, \varphi_i, s) \theta_i \rho + \lambda b(s) \int F(\varphi,s)H_2\rho+\tau(s)(M-\lambda b^2)\\
&\qquad -\underset{i}{\sum} \int B_i(y, s) \theta_i \rho + \lambda b(s) \int B(y,s)H_2\rho.
\end{aligned}
\end{equation}
To estimate the first integral term (the second was already estimated in \eqref{(5.29)}), we first note that, 
by~\eqref{(5.14)},
for large $s$, we have $|F_i(\varphi, \varphi_i,s)| \leq C | \varphi \varphi_i |$. 
Then, for each $\eta > 0$, we write
\begin{equation*} \left| \int F_i(\varphi, \varphi_i,s)\theta_i \rho \right| \leq C\int | \varphi \varphi_i \theta_i | \rho \leq \eta \int (| \nabla \theta |^{2} + | \nabla \varphi |^{4}) \rho + C_\eta \int \varphi^{4} \rho.  \end{equation*}
Using the boundedness of $\varphi, \nabla \varphi$ (owing to.~\eqref{(1.9)}), \eqref{(5.24)},
\eqref{JKL} and \eqref{(5.32)}, 
we deduce that $ | \int F_i(\varphi, \varphi_i,s)\theta_i \rho | \leq \eps(s)(N+b^{2}).$ 
To control the term involving $B_i$, we use \eqref{boundsB} to write
$$
	\left| \int  B_i(y,s)\theta_i \rho \right| \le
\frac{1}{16} 
 M(s)+C\int |\nabla B(y,s)|^2\rho\le  \frac{1}{16} M(s)+Ce^{-2s}. $$
From \eqref{(5.27)}, \eqref{(5.29)}, \eqref{B_1}, \eqref{(5.32)} and \eqref{(5.33)}, we then obtain
$$\begin{aligned}
 \frac{1}{2}(M-\lambda b^{2})' 
 &\leq -N + \frac{1}{2}M + \eps(s)(N+ b^{2}) +\tau(s)(M-\lambda b^2)
 + \frac{1}{16} M(s)+ \frac18 \bigl(\Theta(s)+\lambda b^2(s)\bigr)+Ce^{-2s}\\
 &\leq -\frac{1}{4}(M-\lambda b^{2}) +Ce^{-2s},
 \end{aligned}$$
and we easily deduce that $M \leq \lambda b^{2} + Ce^{-s/2}$ as $s \to \infty$. 
By \eqref{compb2} 
we deduce that $ M = o(b^{2})$ i.e., \eqref{(5.30)}.
\smallskip

\textit{\textbf{Step 6.} Computation of the decay rate of $b$  and  convergence  in} $H^{1}_\rho$. We shall show that 
\begin{equation}\label{(5.34)} \underset{s\to \infty}{\lim}sb(s) = \frac{1}{4pc_2}.\end{equation}
Note that, owing to \eqref{(5.26)} and Step 3 and 4, we have $| a | = o(| b |)$, hence, by Step 5, 
\begin{equation}\label{(5.35)} \| \varphi (s) - b(s)H_2\|_{H^{1}_\rho }= (a^{2}(s) + \| \theta (s)\|^{2}_{H^{1}_\rho})^{1/2} = o(| b(s) |).\end{equation}
 Since $H_2 = c_2(|y|^2 - 2n)$, property \eqref{(5.34)} will thus guarantee the $H^{1}_\rho$ convergence in the statement of the theorem; cf \eqref{(5.4)}. 

To prove \eqref{(5.34)}, going back to \eqref{(5.18)}, we compute $\frac{b'}{b^{2}} =\frac{\tau(s)}{b(s)} 
 -\frac{\gamma(s)}{b^2(s)} +\int V(\varphi,s)\bigl(\frac{\varphi(y,s)}{b}\bigr)^{2}H_2\rho dy$, where 
$V(\xi,s)=\frac{F(\xi,s)}{\xi^2}$ and $\gamma(s):=\int{B(y,s)}H_2\rho$. 
Recall from \eqref{cvFV}, \eqref{limLprime} and  \eqref{limhbeta} that $V(\xi,s)$ extends by continuity as $\xi\to 0$ and that
\be{Vinfty}
V_\infty:=\lim_{s\to\infty}V(0,s)= 
-\frac{p}{2}.
\ee
We have $\lim_{s\to\infty}\frac{b(s)}{\|\varphi(s)\|}=1$ owing to \eqref{(5.19)}, \eqref{(5.25)} and,
by \eqref{lowerdecayw}, \eqref{W0s} and \eqref{boundsB}, we deduce that 
\be{tauinfty}
\lim_{s\to\infty}\Bigl(\frac{\tau(s)}{b(s)}-\frac{\gamma(s)}{b^2(s)}\Bigr)=0.
\ee
Since $\varphi(s) \approx b(s)H_2$, we expect that $\frac{b'}{b^{2}} \approx \ell:=V_\infty\int H^{3}_2 \rho.$ 

To justify this rigorously, we proceed as follows. Fix any $R>0$ and observe that,
owing to \eqref{limLprime}, \eqref{limhbeta}, \eqref{(5.13)}, \eqref{TaylorRI}, \eqref{TaylorRI3}
and \eqref{ell11}, we have 
$\tilde{m}_R(s) := \sup_{|y| \leq R} |V(0,s) - V(\varphi(y,s),s)| \to 0$ as $s \to \infty$.
We write
\begin{equation*}
\begin{aligned} 
\frac{b'}{b^{2}}(s) &-V(0,s)\int H^{3}_2 \rho- \frac{\tau(s)}{b(s)}+\frac{\gamma(s)}{b^2(s)}\\ 
&= \int (V(\varphi,s) - V(0,s))\left(\frac{\varphi}{b}\right)^{2}H_2\rho + V(0,s) \int \left( \left(\frac{\varphi}{b}\right)^{2} -H^{2}_2\right)H_2\rho \equiv T_1 + T_2.
\end{aligned} 
\end{equation*}
Let us first estimate $T_1$. Setting $\rho_1 = (1 + |y|^{2})\rho$, using the boundedness of 
$V(\varphi,s)$ (cf. \eqref{(5.14)}), Proposition~\ref{prop1}, and the Cauchy-Schwarz inequality, we get
\begin{align*} 
|T_1|
&\leq \int_{| y | \leq R} | V(0,s) - V(\varphi,s))|\left(\frac{\varphi}{b}\right)^{2}| H_2|\rho +  C\int_{| y | \geq R} \left(\frac{\varphi}{b}\right)^{2}| H_2|\rho \\
&\leq C\tilde{m}_R(s)\int_{| y | \leq R} \left(\frac{\varphi}{b}\right)^{2}\rho_1 + C\int_{| y | > R} \left| \left(\frac{\varphi}{b}\right)^{2}- H^{2}_2 \right|\rho_1 + C\int_{| y | > R} | H_2 |^{3}\rho\\
&\leq C\frac{\tilde{m}_R(s)}{b^{2}}\| \varphi \|^{2}_{H^{1}_\rho} + C\left \| \frac{\varphi}{b} - H_2 \right \|_{H^{1}_\rho} \left(\left \| \frac{\varphi}{b}\right \|_{H^{1}_\rho} + \| H_2 \|_{H^{1}_\rho}\right) + C\int_{| y | > R} | H_2 |^{3}\rho.
\end{align*}
Next, by Proposition~\ref{prop1} and the Cauchy-Schwarz inequality, we see that $| T_2 |$ is bounded by the second term in the last line of the last inequality. Letting $s\to \infty$ and using \eqref{(5.35)}, \eqref{Vinfty}, \eqref{tauinfty} we thus obtain 
$\overline{\underset{s\to \infty}{\lim}} \bigl| \frac{b'}{b^{2}}(s) -\ell \bigr| \leq C\int_{| y | > R} | H_2 |^{3}\rho$; hence $\underset{s\to \infty}{\lim} \frac{b'}{b^{2}}(s) =\ell$
by letting $R \to \infty$. After integration, we end up with $\underset{s\to \infty}{\lim}sb(s)= -\ell^{-1}$. Finally, by a straightforword calculation, using \eqref{Vinfty}, 
we see that $\ell =-4pc_2$.
\smallskip

\textit{\textbf{Step 7.} Convergence in} $L^{\infty}_{loc}.$ Going back to \eqref{(5.17)} and using \eqref{(5.18)}, we write
\begin{equation*}
\theta_s + \mathcal{L} \theta = \theta+ \tau(s)\theta + F(\varphi,s) -B(y,s)
- c^{2}_0 \int (F(\varphi,s)-B(y,s))\rho - H_2(y)\int (F(\varphi,s)-B(y,s))H_2\rho.
\end{equation*}
Fix $R_0>0$. For all $| y | \leq R_0$, owing to \eqref{(1.9)},
 \eqref{W0s}, \eqref{(5.34)} and \eqref{(5.35)}, 
we have $| \theta | = | \varphi - a(s)H_0 - b(s)H_2| \leq C(R_0)$ and
  $|\tau(s)\theta|\le C|\theta|$, as well as
\begin{equation*}
| F(\varphi,s) | \leq C(\theta + aH_0 + bH_2)^{2} \leq C\theta^{2}  + C(aH_0 + bH_2)^{2} \leq C(R_0)(|\theta | + s^{-2}).
\end{equation*}
Using Proposition~\ref{prop1}, \eqref{suppB} and Step~6, we then obtain, for all $s_2 >s_1$ 
and 
$(y,s) \in Q(s_2) := B_{R_0} \times [s_2, s_2+1]$,
\begin{equation*}| \theta_s + \mathcal{L}\theta | \leq C(R_0)(|\theta | + s_2^{-2}) + C\| \varphi(s) \|^{2}_{H^{1}_\rho}(1 + | y |^{2}) \leq C(R_0)(|\theta | + s_2^{-2}).\end{equation*}
Now fix $\eps > 0$. By Step 4, for any $s_2$ sufficiently large, we have $\| \theta\|_{L^{m}(Q(s_2))} \leq C(R_0)\eps s_2^{-1}$ with $m=2$. By interior parabolic $L^{q}$ estimates and a simple bootstrap argument on $m$, we can then show that this remains true for $m=\infty$.
\end{proof}

\section{Appendix: Known facts and auxiliary results}

 We first recall some known results (valid for not necessarily radial solutions),
which are used in the proofs of Theorems~\ref{theo} and \ref{Thm1}
(see \cite[Theorem~3.1]{souplet2022universal} 
and \cite[Theorem~5.1 and Proposition~3.2]{chabi1} for part (i), and \cite[Theorem~2.1 and Proposition~2.1]{chabi1} for part~(ii)).
We note that hypothesis~\eqref{sem2} is implied by our assumptions~\eqref{sem0}-\eqref{sem} (cf.~before \eqref{limLprime}).

\begin{thm}\label{RDT}
	Let $\Omega\subset \R^n$ be a smooth domain and $p\in(1,p_S)$.
	 Assume \eqref{sem0} and set $L(s):=\frac{f(s)}{s^p}$.
	 Let $u_0\in L^\infty(\Omega)$, with $u_0\ge 0$, be such that $T:=T_{\max} (u_0)<\infty$.
	 
	 \smallskip
	 
	(i) If 
	 $\lim_{s\to\infty} \frac{sL'(s)}{L(s)}=0$, then there exist $M_0>0$ and $t_0\in(0,T)$ such that
	 \be{type1b}
	 \|u(t)\|_\infty\le M_0\psi(t)
	 \quad\hbox{and}\quad \|\nabla u(t)\|_\infty\le M_0\frac{\psi(t)}{\sqrt{T-t}},\quad t_0<t<T.
	 \ee
	 
	(ii) If
	 \be{sem2}
\frac{sL'(s)}{L(s)}=O\bigl(\log^{-\alpha}(s)\bigr)
\hbox{ for some $\alpha>\frac12$,} \quad
\lim_{s\to\infty} \frac{s^2L''(s)}{L(s)}=0,
\ee 
 and $a\in \Omega$ is a blow-up point of $u$, then 
\be{jojo0}
	\lim_{t\to T}\frac{u(a+y\sqrt{T-t},t)}{\psi(t)}=  1,
\ee
	uniformly on compact sets $|y|\le C$. Moreover,
  \be{noneedle}
\lim_{(x,t)\to (a,T)} u(x,t)=\infty.
  \ee
\end{thm}

\smallskip

We next prove the asymptotic properties of the functions $G, H$ that we used repeatedly
(Lemmas~\ref{lm1} and~\ref{lm1b}).

\begin{proof}[Proof of Lemma~\ref{lm1}]
We may fix $s_0>0$ such that $L(s)>0$ for all $s\ge s_0$.

 \smallskip
 
(i) Let $X>s_0$. 
For all $Y>X$, integrating by parts, we obtain
\be{G-IPP}
\int_X^Y \frac{s^{-p}}{L(s)}ds
=\Bigl[\frac{-s^{1-p}}{(p-1)L(s)}\Bigr]_X^Y-\int_X^Y \frac{s^{1-p}L'(s)}{(p-1)L^2(s)}ds.
\ee
By \eqref{hyplm1}, it follows that
$$\left|\int_X^Y \frac{s^{-p}}{L(s)}ds+\Bigl[\frac{s^{1-p}}{(p-1)L(s)}\Bigr]_X^Y\right|
\le \eps(X)\int_X^Y \frac{s^{-p}}{L(s)}ds$$
where $\lim_{s\to\infty} \eps(s)=0$.
Since \eqref{hyplm1} implies 
\be{sloweta}
\frac{1}{L(s)}=o(s^\eta)\ \ \hbox{as $s\to\infty$,\ for any $\eta>0$,}
\ee
we may let $Y\to\infty$, to get
$\left|G(X)-\frac{X^{1-p}}{(p-1)L(X)}\right|
\le \eps(X)G(X)$, hence \eqref{ch}.

\smallskip
 Next setting $\tilde f(s)=\frac{f(s)}{A+\log f(s)}$ and 
$\tilde L(s)=s^{-p}\tilde f(s)=\frac{L(s)}{A+\log f(s)}=\frac{L(s)}{A+p\log s+\log L(s)}$, we compute
$$\tilde L'(s)=\frac{L'(s)}{A+\log f(s)}\Bigl(1-\frac{1}{A+\log f(s)}\Bigr)-\frac{pL(s)}{s\bigl(A+\log f(s)\bigr)^2},$$
hence 
$$\frac{s\tilde L'(s)}{\tilde L(s)}=\frac{sL'(s)}{L(s)}\Bigl(1-\frac{1}{A+\log f(s)}\Bigr)-\frac{p}{A+\log f(s)}.$$
Consequently, $\tilde L$ satisfies property \eqref{hyplm1}, so that \eqref{ch1} follows from \eqref{ch} with $\tilde L$ instead of $L$.

\smallskip

Finally, combining \eqref{ch}, \eqref{ch1} with
the fact that $|\log G(X)|\sim (p-1)\log X$ as $X\to\infty$ 
we obtain~\eqref{ch2}.

 \smallskip
 
(ii) Let $\epsilon\in (0,1/2)$. Since $L$ has slow variation at $\infty$, we easily deduce from \eqref{ch} the existence of $\eta,X_0>0$ such that
$$G((1+\eta)X)\ge(1-\epsilon)G(X)\quad\hbox{and}\quad G((1-\eta)X)\le(1+\epsilon)G(X)\quad\hbox{for all $X\ge X_0$.}$$
Since $G$ is decreasing at $\infty$, this implies property \eqref{equivGG} for $G$ and the same applies to $H$.

Likewise, \eqref{ch} implies the existence of $\bar\eta,\bar X_0>0$ such that,  for all $X\ge X_1$,
$G((1-\epsilon)X)\ge(1+\bar\eta)G(X)$ and $G((1+\epsilon)X)\le(1-\bar\eta)G(X)$.
For $Y\in(0,Y_1]$ with $Y_1>0$ sufficiently small, applying this to $X=G^{-1}(Y)\ge X_1$ 
and using the decreasing monotonicity of $G$, we obtain
$(1-\epsilon)G^{-1}(Y)\le G^{-1}((1+\bar\eta)Y)$ and $(1+\epsilon)G^{-1}(Y)\ge G^{-1}((1-\bar\eta)Y)$.
This implies property \eqref{equivGG} for $G^{-1}$ and the same applies to $H^{-1}$.

 \smallskip
 
 (iii) For all $Y>0$ sufficiently small we may set $X=G^{-1}\bigl({(p-1)Y\over p|\log Y|}\bigr)$ ($\to\infty$ as $Y\to 0^+$),
so that ${(p-1)Y\over p|\log Y|}=G(X)$.
As $Y\to 0^+$, we have $\log G(X)\sim \log Y$ so that, owing to \eqref{ch2},
    $$Y={p\over p-1}G(X)|\log Y|\sim {p\over p-1}G(X)|\log G(X)|\sim H(X).$$
    Applying property \eqref{equivGG} for $H^{-1}$, we get \eqref{dam}.
\end{proof}

\begin{proof}[Proof of Lemma~\ref{lm1b}]
Integrating twice by parts, for $X$ large enough, we may compute $G(X)$ as follows:
\begin{align*}
(p-1)G(X)&=(p-1)\int_{X}^{\infty}\frac{ds}{f(s)}
=(p-1)\int_{X}^{\infty}\frac{ds}{s^pL(s)}=\frac{X^{1-p}}{L(X)}-\int_{X}^{\infty}\frac{s^{1-p}L'(s)}{L^2(s)}ds\\
&=\frac{X^{1-p}}{L(X)}+\int_{X}^{\infty}G'(s)\frac{sL'(s)}{L(s)}ds
=\frac{X^{1-p}}{L(X)}-\frac{XL'(X)}{L(X)}G(X)
-\int_{X}^{\infty}G(s)\Bigl(\frac{sL'(s)}{L(s)}\Bigl)'ds,
\end{align*}
hence
$$\bigl((p-1)X^{p-1}L(X)+X^pL'(X)\bigr)G(X)-1=\mathcal{J}(X):=-X^{p-1}L(X)\int_{X}^{\infty}G(s)\Bigl(\frac{sL'(s)}{L(s)}\Bigl)'ds.
$$
By the second part of \eqref{sem} and \eqref{ch}, we have 
$$\mathcal{J}(X)=o\bigl(X^{p-1}L(X)\mathcal{\tilde J}(X)\bigr)\ \ \hbox{as $X\to\infty$,\quad where }
\mathcal{\tilde J}(X):=\int_{X}^{\infty}\frac{s^{-p}}{L(s)\log s}ds.$$
 To estimate $\mathcal{\tilde J}(X)$, for all $Y>X$, integrating by parts, we write
$$\int_X^Y \frac{s^{-p}}{L(s)\log s}ds
=\Bigl[\frac{-s^{1-p}}{(p-1)L(s)\log s}\Bigr]_X^Y-\int_X^Y \frac{s^{-p}}{p-1}\frac{s(L(s)\log s)'}{(L(s)\log s)^2}ds.$$
Next using $\lim_{s\to\infty} \frac{s(L(s)\log s)'}{L(s)\log s}
=\lim_{s\to\infty} \bigl(\frac{sL'(s)}{L(s)}+\frac{1}{\log s}\bigr)=0$, we get
$$\left|\int_X^Y \frac{s^{-p}}{L(s)\log s}ds+\Bigl[\frac{s^{1-p}}{(p-1)L(s)\log s}\Bigr]_X^Y\right|
\le \eps(X)\int_X^Y \frac{s^{-p}}{L(s)\log s}ds.$$
By \eqref{sloweta}, we may let $Y\to\infty$ to deduce
$\mathcal{\tilde J}(X)\sim\beta\frac{X^{1-p}}{L(X)\log X}$,
hence $\mathcal{J}(X)=o(1/\log X)$, as $X\to\infty$. This proves the lemma.
\end{proof}

We finally give the further asymptotic properties of $f, G$ and $G^{-1}$, 
which were used in the last step of the proof of the lower estimate of Theorem~\ref{theo}, and to obtain the explicit formula \eqref{asymptGinv2} for the global blow-up profile.

\begin{lem}\label{ell12a}
	Let $f$ satisfy \eqref{sem0}-\eqref{sem} and $G$ be defined by \eqref{so}.
	\smallskip
	
	(i) There exist $X_1,C_1>0$ such that, for all $\eps\in(0,1/2]$ and $X\ge X_1$,
		\be{ell12a1}
		f((1-\eps)X)\ge (1-\eps)^{p+1}f(X)
		\ee
		and
		\be{ell12a2}
	G((1-\eps)X)\le (1+C_1\eps)G(X).
	\end{equation}
	
	(ii) We have
$$
	G^{-1}(Y)\sim \kappa Y^{-\beta}L^{-\beta}\big(Y^{-\beta}\big),\quad Y\to 0^+.
$$
\end{lem}

For the proof, we need the following technical result:

\begin{lem}\label{ell8}
Under assumptions \eqref{sem0}-\eqref{sem} with $\alpha\in(0,1)$, there exists  $s_1>0$ such that, for all $s>s_1$,
\begin{equation}\label{ell11}
	\Bigl|\frac{L(\lambda s)}{L(s)}-1\Bigl| \ \le 4\frac{|\log \lambda|}{\log^\alpha{\hskip -2.5pt}s},
	\quad\hbox{ for all $\lambda \in I_s:=\bigl[\exp(-\frac18\log^\alpha{\hskip -2.5pt}s),\exp(\frac18\log^\alpha{\hskip -2.5pt}s)\bigr].$}
	\end{equation}
\end{lem}
\begin{proof} 
By assumptions \eqref{sem0}-\eqref{sem},  there exists $s_0>1$ such that 
\be{ell8.1}
\frac{|L'(s)|}{L(s)}\le \frac{1}{s\log^\alpha{\hskip -2.5pt}s}\quad\hbox{ for all $s\ge s_0$.}
\ee
Also, since $s\exp(-\frac18\log^\alpha{\hskip -2.5pt}s)\to \infty$ 
as $s\to\infty$, there exists $s_1>s_0$ such that $s\exp(-\frac18\log^\alpha{\hskip -2.5pt}s)>s_0$ for all $s\ge s_1$. 
It suffices to prove that, for each $\eps\in(0,\frac12)$, there holds
\be{ell8.20}
\Bigl|\frac{L(\lambda s)}{L(s)}-1\Bigl|\ \le \eps+4\frac{|\log \lambda|}{\log^\alpha{\hskip -2.5pt}s}
\quad\hbox{ for all $\lambda\in I_s$.}
\ee

Fix $s\ge s_1$ and assume for contradiction that there exists $\eps\in(0,\frac12)$ such that \eqref{ell8.20} fails.
Since the inequality in \eqref{ell8.20} is true for $\lambda$ close to $1$ by continuity of $L$, 
there must exist $\lambda_0\in I_s\setminus\{1\}$ such that
\be{ell8.2}
\Bigl|\frac{L(\lambda s)}{L(s)}-1\Bigl|\ \le \eps+4\frac{|\log \lambda|}{\log^\alpha{\hskip -2.5pt}s}\ \hbox{ for all $\lambda \in (a,b)$\quad  and }\quad
	\Bigl|\frac{L(\lambda_0 s)}{L(s)}-1\Bigl|\ =\eps+4\frac{|\log \lambda_0|}{\log^\alpha{\hskip -2.5pt}s},
	\ee
	where either $a=\lambda_0<b=1$ or $a=1<b=\lambda_0$.
		  On the other hand, using the elementary inequality $|(1+h)^{1-\alpha}-1|\le 2(1-\alpha)|h|$ for $\alpha\in(0,1)$ and $h\in(-1,1)$, we compute
		 	$$\begin{aligned}
			\int_{as}^{bs}\frac{dz}{z\log^\alpha z}
			&=(1-\alpha)^{-1}\bigl|\log^{1-\alpha}(bs)-\log^{1-\alpha}(as)\bigr|
			 =(1-\alpha)^{-1}\bigl|\bigl(\log s\pm|\log\lambda_0|\bigr)^{1-\alpha}-\log^{1-\alpha}s\bigr| \\
		 	&=(1-\alpha)^{-1}\log^{1-\alpha}s\Bigl|\Bigl(1\pm\frac{|\log\lambda_0|}{\log s}\Bigr)^{1-\alpha}-1\Bigr|
				 \le 2\frac{|\log\lambda_0|}{\log^\alpha{\hskip -2.5pt}s}.
				 				 	 	\end{aligned} $$
It then follows from \eqref{ell8.1}, \eqref{ell8.2} 
and $\lambda_0\in I_s\setminus\{1\}$ that
	$$\begin{aligned} 
	4\frac{|\log \lambda_0|}{\log^\alpha{\hskip -2.5pt}s}
	 &<\Bigl|\frac{L(\lambda_0 s)}{L(s)}-1\Bigl|\ 
	 =\frac{|L(\lambda_0 s)-L(s)|}{L(s)}
	 \le  \int_{as}^{bs} \frac{|L'(z)|}{L(s)}dz
	 	 \le  \int_{as}^{bs} \frac{L(z)}{L(s)}\frac{dz}{z\log^\alpha z}\\
		 &\le \Bigl(1+\eps+4\frac{|\log \lambda_0|}{\log^\alpha{\hskip -2.5pt}s}\Bigl)\int_{as}^{bs}
		 \frac{dz}{z\log^\alpha z}
		 		 < 2\Bigl(\frac32+4\frac{|\log \lambda_0|}{\log^\alpha{\hskip -2.5pt}s}\Bigl)\frac{|\log\lambda_0|}{\log^\alpha{\hskip -2.5pt}s}
		 \le 4\frac{|\log\lambda_0|}{\log^\alpha{\hskip -2.5pt}s}:	\end{aligned} $$
		a contradiction.
		\end{proof}

\begin{proof}[Proof of Lemma~\ref{ell12a}]
(i) By \eqref{ell11}, there exists $X_1>0$ such that, for all $\eps\in(0,1/2)$ and $X\ge X_1$, 
	$$f((1-\eps)X)=(1-\eps)^pX^pL((1-\eps)X)\ge (1-\eps)^pX^pL(X)\bigl(1-|\log(1-\eps)|\bigr)\ge (1-\eps)^{p+1}f(X),$$
hence \eqref{ell12a1}. It then follows that
	$$G((1-\eps)X)=\int_{(1-\eps)X}^\infty \frac{ ds}{f(s)}=(1-\eps)\int_{X}^\infty \frac{ dz}{f((1-\eps)z)}
	\le (1-\eps)^{-p}G(X),$$
	hence \eqref{ell12a2}.

\smallskip

(ii)
We may assume $\frac{1}{2}<\alpha<1$ without loss of generality. 
Using \eqref{ch} in Lemma~\ref{lm1} and setting $X=X(Y):=G^{-1}(Y)$, we write
\be{equivYbeta}
Y^{-\beta}=(1+\eps(X))\kappa^{-1} XL^\beta(X),\quad\hbox{with \ $\lim_{X\to\infty} \eps(X)=0$.}
\ee
 We shall apply Lemma~\ref{ell8} with $\lambda=\lambda(X):=\kappa^{-1} (1+\eps(X))L^\beta(X)$. 
For all $X\ge s_0$, with $s_0$ as in \eqref{ell8.1}, we have
	\begin{align*}
	|\log (L^\beta(X))|&\le |\log (L^\beta(s_1))|+ \beta \int^X_{s_0}\frac{|L'(s)|}{L(s)}ds
	\le C+\beta \int^X_{s_0}\frac{ds}{s\log^\alpha{\hskip -2.5pt}s}
	\le C+\frac{\beta \log^{1-\alpha} X}{1-\alpha}.
	\end{align*}
	Since $\alpha>\frac12$, we deduce that 
	$ \lim_{X\to\infty} \frac{|\log(L^\beta(X))|}{\log^\alpha X}=0$,
	hence $ \lim_{X\to\infty} \frac{|\log \lambda(X)|}{\log^\alpha X}=0$, and it follows from \eqref{ell11} that 
	$L(Y^{-\beta})=L(\lambda(X)X)\sim L(X)$ as $X\to\infty$. Going back to \eqref{equivYbeta}, we obtain 
	$G^{-1}(Y)=X\sim \kappa Y^{-\beta}L^{-\beta}\big(Y^{-\beta}\big)$ as $Y\to 0^+$, 
	which is the  desired result.
\end{proof}

\noindent{\bf Acknowledgement.} We thank the referee for careful reading	which helped us improve the presentation.

\medskip

\noindent{\bf Statements and Declarations.} 
On behalf of all authors, the corresponding author states that there is no conflict of interest. 
This manuscript has no associated data.

\end{document}